\documentclass[10pt]{article}

\usepackage{amssymb}
\usepackage{amsmath}
\usepackage{theorem}
\usepackage{epsfig}
\usepackage{verbatim}
\usepackage{graphicx}
\usepackage{subfigure}
\usepackage{epstopdf}
\usepackage{mathrsfs}  
\usepackage{hyperref}

\usepackage{color}

\newtheorem{theorem}{Theorem}[section]
\newtheorem{lemma}[theorem]{Lemma}

\newtheorem{prop}[theorem]{Proposition}
\newtheorem{corollary}[theorem]{Corollary}

\newtheorem{defi}[theorem]{Definition}
\newtheorem{rem}[theorem]{Remark}

\newcommand{\R}{\mathbb{R}}

\newcommand{\T}{\mathbb{T}}

\newcommand{\ep}{\varepsilon}

\newenvironment{proof}{\begin{trivlist} \item[] {\em Proof:}}{\hfill $\Box$
                       \end{trivlist}}

\newenvironment{proofthm}[1]{\begin{trivlist} \item[] {\em Proof of Theorem \ref{#1}:}}{\hfill $\Box$
                       \end{trivlist}}
                       
\newenvironment{proofprop}[1]{\begin{trivlist} \item[] {\em Proof of Proposition \ref{#1}:}}{\hfill $\Box$
                       \end{trivlist}}
                                              
\newenvironment{prooflem}[1]{\begin{trivlist} \item[] {\em Proof of Lemma \ref{#1}:}}{\hfill $\Box$
                       \end{trivlist}}

\makeatletter
\renewcommand*\l@section{\@dottedtocline{1}{0em}{1.5em}}
\renewcommand*\l@subsection{\@dottedtocline{2}{1.5em}{2.3em}}
\renewcommand*\l@subsubsection{\@dottedtocline{3}{3.8em}{3.7em}}
\makeatother

\numberwithin{equation}{section}

\allowdisplaybreaks[4]

\begin{document}

\title{Quantitative estimates for uniformly-rotating vortex patches}
\author{Jaemin Park}

\date{}

\maketitle

\begin{abstract}
In this paper, we derive some quantitative estimates for uniformly-rotating vortex patches. We prove that if a non-radial simply-connected patch $D$ is uniformly-rotating with small angular velocity $0 < \Omega \ll 1$, then the outmost point of the patch must be far from the center of rotation, with distance at least of order $\Omega^{-1/2}$. For $m$-fold symmetric simply-connected rotating patches, we show that their angular velocity must be close to $\frac{1}{2}$ for $m\gg 1$ with the difference at most $O(1/m)$, and also obtain estimates on $L^{\infty}$ norm of the polar graph which parametrizes the boundary.
  \end{abstract}

\section{Introduction}
Let us consider the two-dimensional incompressible Euler equation in vorticity form:

\begin{align}\label{Euler}
\begin{cases}
\partial_t \omega + \vec{u}\cdot \nabla \omega = 0 & \text{ in }\R^2\times \R_+, \\
\vec{u}(\cdot, t) = - \nabla^{\perp} (-\Delta)^{-1}\omega (\cdot, t) & \text{ in }\R^2, \\
\omega(\cdot,0)=\omega_0 & \text{ in }\R^2,
\end{cases}
\end{align}
where $\nabla^{\perp}:=(-\partial_{x_2},\partial_{x_1})$. The velocity vector $\vec{u}$ can be recovered from the vorticity $\omega$ by the  Biot--Savart law, namely, 
\begin{align*}
\vec{u}(x,t) = \nabla^{\perp} \left( \omega * \mathcal{N}\right)(x,t) = \frac{1}{2\pi}\int_{\R^2} \omega(y,t)\frac{(x-y)^{\perp}}{|x-y|^2}dy,
\end{align*}
where $\mathcal{N}(x):=\frac{1}{2\pi}\log|x|$ is the Newtonian potential in two dimensions. A weak solution for \eqref{Euler} of the form $\omega(x,t) = 1_{D_t}(x)$ for some bounded domain $D_t$ is called a vortex patch. For the well-posedness results for vortex patches, we refer to \cite{bertozzi1993global, chemin1993persistance,elgindi2020singular,elgindi2019singular,yudovich1963non}.

 A vortex patch $1_{D_t}$ is said to be uniformly-rotating (about the origin) with a constant angular velocity $\Omega$ if the time-dependent domain $D_t$ is given by a rotation of the initial domain, that is,
  \begin{align*}
 D_t= \left\{ x \in \R^2 :  R_{\Omega t}x\in D_0\right\},
 \end{align*}
where $R_{\Omega t}$ denotes the rotation matrix, 
\begin{align*}
R_{\Omega t} := 
\begin{pmatrix}
\cos(\Omega t) & -\sin(\Omega t) \\
\sin(\Omega t) & \cos(\Omega t)
\end{pmatrix}.
\end{align*} 
 Clearly, any radial vortex patch (e.g. a disk or an annulus) is a uniformly-rotating solution with any angular velocity. The first non-radial example was discovered by Kirchhoff in \cite{kirchhoff1876vorlesungen}, namely, he showed that any elliptical patches are uniformly-rotating solutions (see also \cite[Chapter 7]{lamb1924hydrodynamics}).  For other simply-connected patches, Deem--Zabusky \cite{deem1978vortex} numerically found families of rotating patches, having $m$-fold symmetry for some integer $m\ge2$. This result was rigorously proved by Burbea \cite{burbea1982motions}, where it was shown that there are bifurcation curves of $m$-fold symmetric patches, emanating from the unit disk with $\Omega=\frac{m-1}{2m}$. Hmidi--Mateu--Verdera \cite{hmidi2013boundary} showed that the solutions on the bifurcation curves have smooth boundary if they are close enough to the unit disk and the analytic boundary regularity was proved by Castro--C\'ordoba--G\'omez-Serrano in \cite{castro2016uniformly}. Recently, Hassainia--Masmoudi--Wheeler \cite{hassainia2019global} showed that those bifurcation curves can be continued as long as the angular fluid velocity in the rotating frame does not vanish on the boundary, and  it actually becomes arbitrarily small as the parameter of the curve approaches to infinity. This is consistent with the numerical/theoretical evidence of the development of $90^{\circ}$ corners in the limiting patches \cite{overman1986steady,wu1984steady}.  See also \cite{de2016doubly1} for multi-connected rotating patches and \cite{cao2019rotating,de2016analytical} for rotating vortex patches in a bounded domain. We also refer to \cite{castro2019uniformly,castro2020global,de2016doubly,garcia2020non,gomez2019existence,hassainia2015v} for other related questions (e.g. smooth setting, non-uniform density and other rotating active scalars).
 
  Interestingly, the angular velocities of all the non-radial examples found in the previous literature are shown to be in $\left(0,\frac{1}{2}\right)$. Indeed, Frankel \cite{fraenkel2000introduction} proved that any simply-connected stationary patch (in other words, rotating with $\Omega=0$) is necessarily a disk and the same radial symmetry result for $\Omega<0$ was proved by Hmidi \cite{hmidi2015trivial} under some additional convexity assumptions on the patch. Hmidi \cite{hmidi2015trivial} also proved that if a simply-connected vortex patch is uniformly-rotating with $\Omega=\frac{1}{2}$, then it must be a disk. The general case was resolved recently by G\'omez-Serrano--Shi--Yao and the author in \cite{gomez2019symmetry}, where they showed that any uniformly-rotating patch with angular velocity $\Omega\in (-\infty,0]\cup[\frac{1}{2},\infty)$ must be radially symmetric whether it is simply-connected or not (see Figure~\ref{results} for the illustration of radial symmetry results).
  
   \begin{figure}[htbp]
\begin{center}
    \includegraphics[scale=1]{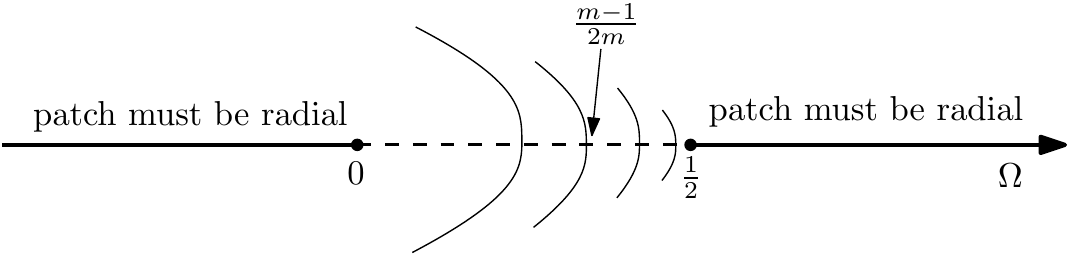}
    \caption{Existence/Non-existence of non-radial solutions depending on $\Omega$. }
   \label{results}
    \end{center}

\end{figure}

\subsection{Main results}
 The goal of this paper is to establish some quantitative estimates for non-radial simply-connected rotating patches, which are known to exist.  From now on, we assume that a bounded domain $D$ is simply-connected and has $C^{2}$ boundary.  If $\omega(x,t) := 1_{D}(R_{\Omega t} x)$ is a uniformly-rotating patch, then the net velocity in the rotating frame has no contribution to the deformation of the boundary $\partial D$, namely,  $\nabla^{\perp}\left((1_D * \mathcal{N})-\frac{\Omega}{2}|x|^2\right)\cdot \vec{n} = 0$, where $\vec{n}$ denotes the outer normal vector on $\partial D$. By integrating this along the boundary, one can derive the following equation for the relative stream function $\Psi$:
 
\begin{align}\label{rotatingpatch}
\Psi(x) := 1_D * \mathcal{N} - \frac{\Omega}{2}|x|^2 = constant \quad \text { for all }x\in\partial D.
\end{align}

In the rest of the paper, we say a pair $(D,\Omega)$ is a solution to \eqref{rotatingpatch} if $1_D$ and $\Omega$ satisfy the equation \eqref{rotatingpatch}.

 \subsubsection{Small angular velocity $\Omega$} 
 Our first main result is about the outmost point on $\partial D$ when the angular velocity $\Omega$ is small. As mentioned earlier, ellipses are uniformly-rotating solutions. More precisely, an ellipse with semi-axes $a,b$ is rotating with angular velocity $\Omega=\frac{ab}{(a+b)^2}$. By imposing $b=\frac{1}{a}$ to keep the area of the patch equal to $\pi$, one can easily see that for any $0<\Omega\le\frac{1}{4}$, there exists an ellipse that is rotating with the given $\Omega$. Moreover, the boundary is stretching as $\Omega$ tends to $0$ in the sense that the length of the major axis is comparable with $\Omega^{-\frac{1}{2}}$. Note that  ellipses are not the only uniformly-rotating solutions for small angular velocities. For example, the existence of secondary bifurcations from  ellipses was numerically observed by Kamm in his thesis \cite{kamm1987shape} and theoretically proved in \cite{castro2016uniformly,hmidi2015bifurcation}. Thus it is a natural question whether every non-radial simply-connected rotating patch with a fixed area and $0 < \Omega \ll 1$ must have its outmost point very far from the origin (center of rotation). In the next theorem, we prove this is indeed true.

\begin{theorem}\label{smallomega}. Let $D\subset \R^2$ be a simply connected domain such that $|D| = |B| = \pi$, where $B$ is the unit disk centered at the origin. Then there exist positive constants $\Omega_0$ and $\kappa_0$ such that if $(D,\Omega)$ is a solution to \eqref{rotatingpatch} with $\Omega\in (0,\Omega_0)$, then either $D = B$, or
\begin{align}\label{estimate1}
\sup_{x\in \partial D}|x| > \kappa_0 \Omega^{-\frac{1}{2}}.
\end{align}

\end{theorem}

\begin{rem}\label{remark}
Note that the power $-\frac{1}{2}$ is sharp since it is achieved by ellipses. Furthermore, one can easily show that \eqref{rotatingpatch} is scaling invariant in the sense that if $(D,\Omega)$ is a solution, then $(D_a,\Omega)$ is also a solution for any $a>0$, where $D_a := \left\{ ax \in \R^2 : x\in D \right\}$. Therefore without the restriction on the size of the patch, \eqref{estimate1} reads as
\begin{align*}
\frac{1}{\sqrt{|D|}}\sup_{x\in \partial D}|x| > \frac{\kappa_0}{\sqrt{\pi}} \Omega^{-\frac{1}{2}}.
\end{align*} 
\end{rem}


\subsubsection{$m$-fold symmetric patches}
It has been known since the work of Burbea \cite{burbea1982motions} that there are $m$-fold symmetric rotating patches for every integer $m\ge2$. From the numerical results \cite{deem1978vortex,hassainia2019global}, it appears that for $m\gg1$, the angular velocity $\Omega$ along the bifurcation curve is very close to $\frac{1}{2}$ (i.e. $0<\frac{1}{2}-\Omega \ll1$ for $m\gg 1$). But there are no such type of quantitative estimates so far.   In the next theorem, we will derive a lower bound of the angular velocity by imposing large $m$. 
\begin{theorem} \label{largem2}
There exist $m_0\ge 2$ and $c>0$ such that if $(\Omega,D)$ is a solution to \eqref{rotatingpatch} and $D$ is simply-connected, non-radial, and $m$-fold symmetric for some $m\ge m_0$ then
\begin{align*}
\frac{1}{2}-\Omega \le \frac{c}{m}.
\end{align*}
\end{theorem}
We emphasize that this theorem holds for a general simply-connected patch, which does not need to lie on the bifurcation curve.

For $m$-fold symmetric solutions on the global bifurcation curves constructed in \cite{hassainia2019global}, we will also estimate the difference between a rotating patch and the unit disk.  To be precise, we will focus on the curves,
\begin{align*}
\mathscr{C}_m:= \left\{ \left( \tilde{u}_m(s),\tilde{\Omega}_m(s)\right) \in C^{2}(\mathbb{T})\times \left(0,\frac{1}{2}\right) : s\in [0,\infty)\right\} \quad \text{ for  $m\ge 2$, }
\end{align*}
that satisfy the following properties (see \cite[Theorem 1.1]{hassainia2019global} for the details):

\begin{itemize}
\item[(A1)] $\tilde{u}_m(s) \in \left\{ u \in C^{2}(\mathbb{T}) : u(\theta) = \sum_{n=1}^{\infty}a_{n}\cos(nm\theta) \ \text{ for some }(a_n)_{n=1}^{\infty}\text{ and }u> -1\right\}$.
\item[(A2)] $(D^{\tilde{u}_m(s)},\tilde{\Omega}_m(s))$ is a solution for \eqref{rotatingpatch}, where $ D^u := \left\{ r(\cos\theta, \sin\theta) \in \R^2 : 0\le r<(1+u(\theta)),\  \theta\in \mathbb{T}\right\}$.

\item[(A3)] $\partial_\theta u(\theta) < 0$ for all $\theta\in (0,\frac{\pi}{m})$, where $u=\tilde{u}_m(s)$. 
\item[(A4)] $(\tilde{u}_m(0),\tilde{\Omega}_m(0)) = \left( 0 ,\frac{m-1}{2m}\right)$.
\end{itemize}

For such curves, we have the following theorem:

\begin{theorem}\label{largem}
Let $\mathscr{C}_{m} :=\left\{ (\tilde{u}_m(s),\tilde{\Omega}_m(s)) \in  C^{2}(\mathbb{T}) \times \left( 0 ,  \frac{1}{2} \right) : \  0\le s < \infty\right\}$ be a continuous curve that satisfies the properties (A1)-(A4). Then there exist constants $c>0$ and $m_0 \ge 3$ such that if $m\ge m_0$, then
 
 \begin{align*}
 \rVert \tilde{u}_m(s) \rVert_{L^{\infty}(\mathbb{T})} \le \frac{c}{m} \text{ for all }s\ge 0.
 \end{align*}

\end{theorem}

 Although each curve emanates from the unit disk, the possibility that $\min_{\theta\in\mathbb{T}}(1+\tilde{u}_m(s))$ tends to $0$ along the curve  has not been completely eliminated (\cite[Theorem~4.6, Lemma~6.6]{hassainia2019global}), while it does certainly happen for ellipses ($m=2$). The significant difference between $m=2$ and $m\ge3$ is that if $m\ge3$, then the stream function ($1_D *\mathcal{N}$) behaves quite nicely, namely, $\frac{|\nabla(1_D*\mathcal{N})|}{|x|}$ is globally bounded (especially near the origin) independently of $D$  (Lemma~\ref{derivatives}. See also \cite{elgindi2016remarks,elgindi2020symmetries}, and the references therein, where global boundedness of gradient of $m$-fold symmetric stream functions  was proved). This will play a crucial role to eliminate the scenario that  $\partial D$ almost touches the origin when $\frac{1}{2}-\Omega$ is sufficiently large compared to $\frac{1}{m}$ in Lemma~\ref{phimbound}.

We summarize the main results in Figure~\ref{diagam1}
 \begin{figure}[htbp]
\begin{center}
    \includegraphics[scale=1]{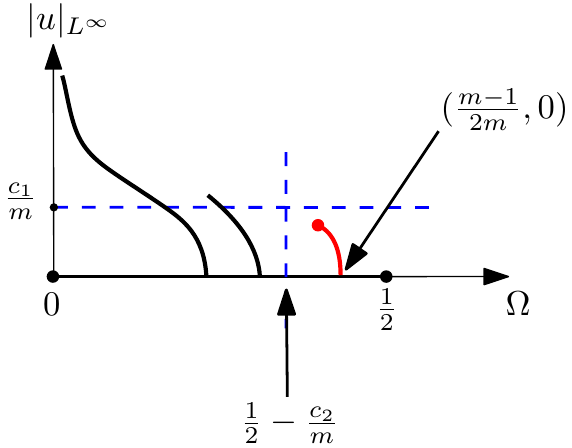}
    \caption{ Illustration of the main results. a) Any patch with $\Omega \ll 1$ must have its outmost point very far from the origin. b) A bifurcation curve of $m$-fold symmetric patches cannot be continued beyond the blue dashed lines for large $m$. }
   \label{diagam1}
    \end{center}

\end{figure}

\paragraph{Idea of the proof} 
The starting point for Theorem~\ref{smallomega} and~\ref{largem2} is the variational formulation of \eqref{rotatingpatch}, used by G\'omez-Serrano, Shi, Yao and the author in \cite{gomez2019symmetry}. Namely, if $(D,\Omega)$ is a solution to \eqref{rotatingpatch} and $\partial D$ is $C^{2}$, then formally, $1_D$ can be thought of as a critical point of the functional,
\begin{align*}
\mathcal{I}(\rho) := \frac{1}{2}\int_{\R^2} \rho * \mathcal{N}(x)\rho(x) dx - \Omega\int_{\R^2} \frac{|x|^2}{2}\rho(x)dx =: \mathcal{I}_1(\rho) - \Omega\mathcal{I}_2(\rho),
\end{align*}
under measure-preserving perturbations. More precisely, it holds that
\begin{align}\label{firstvariation}
 \int_D \vec{v} \cdot \nabla \Psi (x) dx = 0, \quad \text{ for any }v\in C^{2}(\overline{D})\text{ such that } \nabla \cdot \vec{v}=0,
\end{align}
 Indeed, \eqref{firstvariation} follows directly from  \eqref{rotatingpatch} and the integration by parts. By choosing a specific vector field $\vec{v} := x+\nabla p$, where $p$ is defined as the solution to the Poisson equation,
\begin{align}\label{defp}
\begin{cases}
\Delta p = -2 & \text{ in }D, \\
p = 0 & \text{ on }\partial D,
\end{cases}
\end{align}
G\'omez-Serrano  et al. derived the following identity for uniformly-rotating patches:
\begin{align}\label{firstvariation3}
2\Omega\left( \int_{D}\frac{|x|^2}{2}dx-\frac{|D|^2}{4\pi} \right) = (1-2\Omega)\left( \frac{|D|^2}{4\pi}-\int_D p dx\right).
\end{align}

Note that both parentheses are strictly positive if $D\ne B$, where $B$ is the unit disk centered at the origin. Thanks to the result by Brasco--De Philippis--Velichkov in \cite[Proposition 2.1]{brasco2015faber}, one can find a lower bound of the right-hand side of \eqref{firstvariation3} in terms of $|D\triangle B|$, namely, $\frac{|D|^2}{4\pi} - \int_D pdx \gtrsim \left| D \triangle B \right|^2$. Hence \eqref{firstvariation3} yields that
\begin{align*}
| D\triangle B|^2 \lesssim \Omega \left(\int_{D}\frac{|x|^2}{2}dx-\frac{|D|^2}{4\pi}\right) < \Omega\sup_{x\in \partial D}|x|^2,
\end{align*}
for $\Omega \ll 1$.
Therefore we only need to rule out the case where $|D\triangle B|$ is small. Assuming $|D\triangle B|$ and $\Omega$ are sufficiently small, we will prove (Lemma~\ref{starshapelem}) that $D$ is necessarily star-shaped and the boundary can be parametrized by $(1+u(x))x$, for $x\in \partial B_1$ and some $u\in C^{2}(\partial B)$. However, the difficulty is that we have $|D\triangle B|\sim\rVert u \rVert_{L^{1}({\partial B})}$ and $\int_{D}\frac{|x|^2}{2}dx-\frac{|D|^2}{4\pi}\sim \rVert u \rVert_{L^{2}({\partial B})}^2$, while $L^1$ and $L^2$ are not comparable.
  The key idea is to use a different vector $\vec{v} :=x-2\nabla \left(1_D*\mathcal{N}\right)$ in \eqref{firstvariation}, which gives another identity for any simply-connected rotating patches,

\begin{align}\label{firstvariation2}
\left( \frac{1}{2}-\Omega\right) \left( \int_{D}|x|^2dx-\frac{|D|^2}{2\pi} \right) = \frac{1}{2}\int_D | x- 2\nabla \left( 1_D * \mathcal{N} \right) |^2dx.
\end{align}

Thanks to the result of Loeper \cite[Proposition 3.1]{loeper2006uniqueness}, the right-hand side in \eqref{firstvariation2} can be estimated in terms of $2$-Wasserstein distance between $1_D dx$ and $1_B dx$ (see Proposition~\ref{estimate}). In the proof of Proposition~\ref{l1andl2bound}, we will construct an explicit transport map and obtain the bound for the right-hand side: If $\rVert u \rVert_{L^{\infty}(\mathbb{T})} \le \frac{1}{2}$, 
\begin{align}\label{boundforrhs}
\int_D |x - 2\nabla \left( 1_D * \mathcal{N} \right) |^2 dx \lesssim \left( a \int_{\mathbb{T}}|u|^2d\theta+\frac{1}{a} \int_{\mathbb{T}}f(\theta)^2d\theta \right),
 \end{align}
 where $f(\theta) := \int_0^{\theta} u(s)^2+2u(s)ds$ and $a \in (2\rVert u \rVert_{L^{\infty}(\mathbb{T})}, 1)$. Since $\rVert f \rVert_{L^{\infty}(\mathbb{T})} \lesssim \rVert u \rVert_{L^{1}(\mathbb{T})}$, \eqref{firstvariation2} and \eqref{boundforrhs} will give us $\rVert u \rVert_{L^{1}(\mathbb{T})} \sim \rVert u \rVert_{L^{2}(\mathbb{T})}$ for $0<\Omega \ll 1$, if we can choose $a$ sufficiently small.

 The proof of Theorem~\ref{largem2} also relies on the identity \eqref{firstvariation2}.   By imposing $m$-fold symmetry on the patch, we can lower the total cost of the transportation, from which we can obtain a suitable upper bound of $\frac{1}{2}-\Omega$ when $m$ is sufficiently large. Indeed, if $u$ is $\frac{2\pi}{m}$ periodic, then $f$ is also $\frac{2\pi}{m}$ periodic as well. Thus by choosing large $m$, we can lower $\rVert f \rVert_{L^{\infty}(\mathbb{T})}$ on the right-hand side in \eqref{boundforrhs} by using Jensen's inequality.

 Theorem~\ref{largem} will be proved by showing that if $\rVert \tilde{u}_m(s) \rVert_{L^{\infty}}$ is too large, then $\frac{1}{2}-\Omega$ must be large enough to contradict Theorem~\ref{largem2}. The main difficulty is that $\frac{1}{2}-\Omega$ can be estimated in terms of $\rVert \tilde{u}_m(s) \rVert_{L^2(\mathbb{T})}$ by using the identity \eqref{firstvariation3} (Lemma~\ref{lambdaestimate}), while  $L^2$ and $L^{\infty}$  are not comparable. We resolve this issue by estimating the gradient of the stream function in a very delicate way (Lemma~\ref{boundforvarphi2} and ~\ref{derivativeestimates}).

\paragraph{Notations.} In the rest of the paper, we will fix the following notations. We denote by $B_r(x_0)$ the disk, 
 \begin{align*}
 B_r(x_0) := \left\{ x\in \R^2 : | x- x_0 | < r \right\}.
 \end{align*}
 If $r=1$ or $x_0$ coincides with the origin, then we will omit it in the notation. For example, the unit disk centered at $x_0$ is denoted by $B(x_0)$ and  the disk centered at the origin with radius $r$ is denoted by $B_r$. Therefore the unit disk centered at the origin will be  simply denoted by $B$.  For a measurable set $D$ in $\R^2$, we denote the Lebesgue measure of $D$ by $|D|$.  For two domains $D_1$ and $D_2$ in $\R^2$, their symmetric difference is denoted by $D_1 \triangle D_2$, that is,
 \begin{align*}
 D_1 \triangle D_2 := \left( D_1\backslash D_2 \right) \cup \left( D_2\backslash D_1 \right).
 \end{align*}
 
 For a measure $\mu$ in $\R^2$ and a  $\mu$-measurable map $T:\R^2 \mapsto \R^2$, we denote the pushforward measure of $\mu$ by $T_{\#}\mu$, that is, for any $\mu$-measurable set $E$, we have
 \begin{align*}
 T_{\#}\mu(E) = \mu(T^{-1}(E)).
 \end{align*}

For two quantities $X$ and $Y$, we write $X \lesssim Y$ if there is a constant $C>0$ such that $X \le C Y$ where $C$ does not depend on any variables. Furthermore, we shall write $X \sim Y$ if $X \lesssim Y$ and $Y\lesssim X$. 
Lastly, we always assume that $D$ is simply-connected and $\partial D$ is  $C^{2}$.

\section{Quantitative estimates for small $\Omega$}

This section is devoted to the proof of Theorem~\ref{smallomega}. Throughout this section, we will always assume that $|D| = |B| = \pi$. We begin this section by proving two identities for simply-connected rotating patches.

\begin{lemma}\label{firstvariations}
Let $(D,\Omega)$ be a solution to \eqref{rotatingpatch}. Then it holds that
\begin{align}
&\Omega\left( \int_{D}{|x|^2}dx-\frac{|D|^2}{2\pi} \right) = (1-2\Omega)\left( \frac{|D|^2}{4\pi}-\int_D p dx\right), \label{variationalidentity2} \\
&\left( \frac{1}{2}-\Omega\right) \left( \int_{D}|x|^2dx-\frac{|D|^2}{2\pi} \right) = \frac{1}{2}\int_D | x- 2\nabla \left( 1_D * \mathcal{N} \right) |^2dx. \label{variationalidentity1}
\end{align}
where $p$ is the unique solution to \eqref{defp}.
\end{lemma}

 \begin{proof}
  The proof of \eqref{variationalidentity2} can be found in \cite[Theorem 2.2]{gomez2019symmetry}. For the sake of completeness, we give a proof below.
  
In order to prove \eqref{variationalidentity2}, we plug $\vec{v} = x+\nabla p$ into \eqref{firstvariation} to get
\begin{align}
0 &= \int_D (x+\nabla p)\cdot  \nabla \left(1_D * \mathcal{N} - \frac{\Omega}{2} |x|^2\right)dx \nonumber\\
 &= \int_D x \cdot \nabla\left( 1_D * \mathcal{N}\right) - \Omega\int_D |x|^2dx + \int_D \nabla p \cdot \nabla \left( 1_D * \mathcal{N} -\frac{\Omega}{2} |x|^2\right) dx \nonumber\\
 & = \int_D x\cdot \nabla \left( 1_D *\mathcal{N}\right) - \Omega \int_D |x|^2 dx - (1-2\Omega)\int_D p dx, \label{equation91}
\end{align}
where we used divergence theorem for the last equality.  Note that the first integral can be computed as
\begin{align}
\int_{D} x\cdot \nabla \left(1_D * \mathcal{N}\right)dx & = \frac{1}{2\pi}\int_D \int_D \frac{x\cdot (x-y)}{|x-y|^2}dydx = \frac{|D|^2}{4\pi}, \label{equation90}
\end{align}
where the last equality is obtained by exchanging $x$ and $y$ in the double integral, and then taking the average with the original integral. Therefore \eqref{equation91} and \eqref{equation90} yield
\begin{align*}
0 = \frac{|D|^2}{4\pi} - \Omega\int_D |x|^2 dx - (1-2\Omega)\int_D pdx,
\end{align*}
which is equivalent to \eqref{variationalidentity2}. 

 For \eqref{variationalidentity1}, we choose $\vec{v} = x-2\nabla\left( 1_D *\mathcal{N}\right)$ in \eqref{firstvariation} and obtain
 \begin{align}
  0 &= \int_D (x-2\nabla (1_D * \mathcal{N})) \cdot \nabla \left(1_D * \mathcal{N} - \frac{\Omega}{2} |x|^2\right)dx \nonumber \\
   &= (1+2\Omega)\frac{|D|^2}{4\pi} - 2\int_{D} | \nabla \left( 1_D * \mathcal{N}\right) |^2 dx - \Omega\int_D |x|^2dx, \label{equation92}
 \end{align}
 where we used \eqref{equation90}. Since $\int_{D} |\nabla \left( 1_D *\mathcal{N}\right)|^2 dx$ can be computed as
 \begin{align*}
 2\int_{D} |\nabla \left( 1_D *\mathcal{N}\right)|^2 dx  & = \frac{1}{2}\int_D |x-2\nabla\left( 1_D * \mathcal{N}\right)|^2dx + 2\int_D x\cdot \nabla \left(1_D*\mathcal{N}\right)dx - \frac{1}{2}\int_D |x|^2dx \\
  & = \frac{1}{2}\int_D |x-2\nabla\left( 1_D * \mathcal{N}\right)|^2dx + \frac{|D|^2}{2\pi} - \frac{1}{2}\int_D |x|^2dx,
 \end{align*}
 plugging it into \eqref{equation92} yields \eqref{variationalidentity1}. 
\end{proof} 

Thanks to Lemma~\ref{firstvariations}, the angular velocity can be estimated by comparing the quantities, $\int_{D}{|x|^2}dx-\frac{|D|^2}{2\pi}$, $\frac{|D|^2}{4\pi}-\int_D p dx$ and $\int_D | x- 2\nabla \left( 1_D * \mathcal{N} \right) |^2dx$, which vanish if and only if $D=B$. To estimate those quantities for non-radial patches, we use the following notion of asymmetry.

 \begin{defi}\cite[Section 1.1]{fusco2009stability}
 For a bounded domain $D\subseteq \R^2$, the Fraenkel asymmetry $\mathcal{A}(D)$ is defined by
 \begin{align*}
 \mathcal{A}(D):=\inf_{x\in \R^2}\left\{ \frac{|D\triangle B_r(x)|}{|D|} : \pi r^2 = |D| \right\}.
 \end{align*}
 \end{defi}
 
 If $\mathcal{A}(D)$ is not small, then we can find a lower bound of the right-hand side in \eqref{variationalidentity2} by using the following result:

\begin{prop}\cite[Proposition 2.1]{brasco2015faber}\label{alphasquared}
Let $p$ be as in \eqref{defp} and $|D| = \pi$. Then there exists a constant $\sigma>0$ such that
\begin{align}\label{alphasquared1}
\frac{|D|^2}{4\pi} - \int_D pdx \ge \sigma \mathcal{A}(D)^2.
\end{align}
\end{prop}

Using the above proposition and the identity \eqref{variationalidentity2}, one can easily show that $\sup_{x\in \partial D}|x| \gtrsim \sqrt{\mathcal{A}(D)} \Omega^{-\frac{1}{2}}$. Therefore Theorem~\ref{smallomega} can be proved if we can show $\mathcal{A}(D)$ is always bounded below by a strict positive constant. In other words, we will aim to prove in the next lemmas that if $\mathcal{A}(D)$ and $\Omega$ are sufficiently small, then $D$ must be a disk.

  In the following lemma, we will estimate the boundedness of rotating patches in a crude way but this will be improved later.

\begin{lemma}\label{claimford2231}
There exist positive constants $\Omega_1$ and $\alpha_1<\frac{1}{2}$ such that if  $\Omega < \Omega_1$ and $\mathcal{A}(D)<\alpha_1$, then
\begin{align}
D \subset B_2(x_0), \label{claimford1}\\
|x_0| \le 4 \mathcal{A}(D), \label{claimford2}
\end{align}
where $x_0$ is a point such that $\frac{\left| D \triangle B(x_0)\right|}{\pi} = \mathcal{A}(D)$.
\end{lemma}

\begin{proof}
Let us pick $\Omega_1$ and $\alpha_1$ so that for all $\Omega < \Omega_1$ and $\alpha <\alpha_1<\frac{1}{2}$, it holds that
\begin{align}\label{contradiction1}
 \frac{1}{2}\log2 \ge \Omega \left( 10 + \frac{4\alpha}{\left(1-\sqrt{2\alpha}\right)^2}\left(1+\sqrt{\frac{1}{4\Omega}}\right)^2\right) + 3\sqrt{2\alpha}.
\end{align}
We will first show that if $(D,\Omega)$ satisfies $\Omega< \Omega_1$ and $\mathcal{A}(D) <\alpha_1$, then \eqref{claimford1} holds.

 Note that the center of mass of $D$ is necessarily the origin (\cite[Proposition 3]{hmidi2015rotating}). Therefore we have
\begin{align}\label{equation9}
0 = \frac{1}{\pi}\int x1_D(x)dx = \frac{1}{\pi}\int x\left(1_D(x)-1_{B(x_0)}(x)\right)dx + x_0.
\end{align}
Hence it follows from Cauchy-Schwarz inequality that
\begin{align}\label{equation1}
|x_0| \le \frac{1}{\pi} \int |x| \left|1_D(x)-1_{B(x_0)}(x)\right|dx & \le \frac{1}{\pi}\left( \sqrt{|D\triangle B(x_0)|}\sqrt{\int_{D}|x|^2dx + \int_{B(x_0)}|x|^2dx}\right)\nonumber\\
&  \le \sqrt{\frac{\mathcal{A}(D)}{\pi}}\sqrt{\int_{D}|x|^2dx} + \sqrt{\frac{\mathcal{A}(D)}{\pi}}\sqrt{\int_{B(x_0)}|x|^2dx}.
\end{align} 
Since $\sqrt{\int_{B(x_0)} |x|^2dx} \le \sqrt{\int_{B(x_0)} 2|x-x_0|^2dx + 2\int_{B(x_0)}|x_0|^2dx} = \sqrt{\pi + 2|x_0|^2\pi} \le \sqrt{\pi} + \sqrt{2\pi}|x_0|$, \eqref{equation1} yields that
\begin{align}\label{equation2}
\left(1-\sqrt{2\mathcal{A}(D)}\right)|x_0| < \sqrt{\mathcal{A}(D)} + \sqrt{\frac{\mathcal{A}(D)}{\pi}}\sqrt{\int_{D}|x|^2dx}.
\end{align}
In addition, it follows from \eqref{variationalidentity2} that
\begin{align*}
\Omega\int_{D}|x|^2dx = \frac{|D|^2}{4\pi} - (1-2\Omega)\int_{D}pdx < \frac{\pi}{4},
\end{align*}
where we used $\Omega < \frac{1}{2}$, $p\ge 0$ in $D$ and $|D| = \pi$ to get the last inequality. Plugging this into \eqref{equation2}, we obtain
\begin{align*}
\left(1-\sqrt{2\mathcal{A}(D)}\right)|x_0| < \sqrt{\mathcal{A}(D)}\left(1+\sqrt{\frac{1}{4\Omega}}\right),
\end{align*}
hence,
\begin{align}\label{equation3}
|x_0| < \frac{\sqrt{\mathcal{A}(D)}}{1-\sqrt{2\mathcal{A}(D)}}\left(1+\sqrt{\frac{1}{4\Omega}}\right).
\end{align}
 
 To prove \eqref{claimford1}, let us suppose to the contrary that there exist $x_1 \in \partial B(x_0)\cap \partial D$ and $x_2\in \partial B_2(x_0)\cap \partial D$. Then it follows from \eqref{rotatingpatch} that $0 = \Psi(x_1) - \Psi(x_2)$, therefore
\begin{align}\label{equation8}
1_{B(x_0)}*\mathcal{N}(x_2)-1_{B(x_0)}*\mathcal{N}(x_1) = \Omega(|x_2|^2 - |x_{1}|^2) + h(x_1)-h(x_2),
\end{align}
where $h(x) := \left(1_D - 1_{B(x_0)}\right)*\mathcal{N}.$
For the left-hand side, we use 
\begin{align*}
1_{B(x_0)}*\mathcal{N}=
\begin{cases}
\frac{1}{4}|x-x_0|^2 - \frac{1}{4} & \text{ if }|x-x_0|<1, \\
\frac{1}{2}\log|x-x_0| & \text{ otherwise,}
\end{cases}
\end{align*}
and obtain 
\begin{align}\label{equation7}
1_{B(x_0)}*\mathcal{N}(x_2)-1_{B(x_0)}*\mathcal{N}(x_1)  = \frac{1}{2}\log2.
\end{align}
For $\Omega\left(|x_2|^2-|x_1|^2\right)$  in the right-hand side of \eqref{equation8}, we use the triangular inequality and \eqref{equation3} to obtain
\begin{align}
&|x_1|^2 < 2|x_1 - x_0|^2 + 2|x_0|^2 < 2 +  \frac{2\mathcal{A}(D)}{\left(1-\sqrt{2\mathcal{A}(D)}\right)^2}\left(1+\sqrt{\frac{1}{4\Omega}}\right)^2, \label{equation5}\\
&|x_2|^2 < 2|x_2 - x_0|^2 + 2|x_0|^2 < 8 + \frac{2\mathcal{A}(D)}{\left(1-\sqrt{2\mathcal{A}(D)}\right)^2}\left(1+\sqrt{\frac{1}{4\Omega}}\right)^2\label{equation4}.
\end{align}
To estimate $h(x_1) - h(x_2)$, we use the fact that 
 \begin{align}\label{errorterm}
 \rVert \nabla f*\mathcal{N} \rVert_{L^{\infty}}\le  \sqrt{\frac{2}{\pi}}\sqrt{\rVert f\rVert_{L^1}\rVert f\rVert_{L^\infty}} \quad \text{for any $f\in L^1\cap L^{\infty}(\R^2)$}.
 \end{align}
 Indeed, one can compute with $a:=\sqrt{\frac{\rVert f \rVert _{L^1}(\R^2)}{2\pi\rVert f \rVert_{L^{\infty}(\R^2)}}}$, 
 \begin{align*}
 \left| \nabla f * \mathcal{N} (x) \right| \le \frac{1}{2\pi}\int_{|x-y|> a}\frac{1}{|x-y|}\left |f(y) \right|dy + \frac{1}{2\pi}\int_{|x-y| \le a} \frac{1}{|x-y|}\left| f(y) \right| dy \le \frac{\rVert f \rVert_{L^1(\R^2)}}{2\pi a} + a\rVert f \rVert_{L^{\infty}(\R^2)},
 \end{align*}
 which yields \eqref{errorterm}. Thus we have
 \begin{align}\label{equation6}
 |h(x_1) - h(x_2)| < |\nabla h|_{L^{\infty}}|x_1 - x_2| < \sqrt{\frac{2}{\pi}}\sqrt{|D\triangle B(x_0)|}|x_1 - x_2| = 3\sqrt{2\mathcal{A}(D)}.
 \end{align}
 Hence it follows from \eqref{equation8}, \eqref{equation7}, \eqref{equation5}, \eqref{equation4} and \eqref{equation6}  that
 \begin{align}\label{equation10}
 \frac{1}{2}\log2 < \Omega \left( 10 + \frac{4\mathcal{A}(D)}{\left(1-\sqrt{2\mathcal{A}(D)}\right)^2}\left(1+\sqrt{\frac{1}{4\Omega}}\right)^2\right) + 3\sqrt{2\mathcal{A}(D)},
 \end{align}
which contradicts our choice of $\Omega_1$ and $\alpha_1$ for \eqref{contradiction1}. This proves the claim~\eqref{claimford1}. 

To prove \eqref{claimford2},  let us fix $\Omega_1$ and $\alpha_1 < \frac{1}{2}$ so that the claim~\eqref{claimford1} holds. Then for all uniformly-rotating patches with $\Omega < \Omega_1$ and $\mathcal{A}(D)<\alpha_1$, it follows from \eqref{equation9} that 
\begin{align*}
|x_0| \le \frac{1}{\pi}|D\triangle B(x_0)|  \sup_{x\in D \cup B(x_0)}|x| \le \mathcal{A}(D)\left( |x_0| + 2\right),
\end{align*}
where we used \eqref{claimford1} to get the last inequality. Therefore 
\begin{align*}
|x_0| < \frac{2\mathcal{A}(D)}{1-\mathcal{A}(D)} < 4\mathcal{A}(D),
\end{align*}
where we used $\mathcal{A}(D) \le \alpha_1 <\frac{1}{2}$ for the last inequality. Hence \eqref{claimford2} is proved.
\end{proof}

 Since we are interested in patches that rotate about the origin, let us consider the asymmetry between $D$ and the unit disk centered at the origin:
 \begin{align*}
 \mathcal{A}_0(D) : = \frac{|D\triangle B|}{|D|} = \frac{|D \triangle B|}{\pi}.
 \end{align*}

 Tautologically, it holds that $\mathcal{A}(D) \le \mathcal{A}_0(D)$. For rotating patches, we have the following lemma:

\begin{lemma}\label{comparisonasymmetry}
There exist positive constants $\Omega_1$ and $c_1$ such that if $(D,\Omega)$ is a solution to \eqref{rotatingpatch} with $\Omega < \Omega_1$, then 
\begin{align}\label{comparison}
\mathcal{A}_0(D) \le c_1 \mathcal{A}(D).
\end{align}
\end{lemma}
\begin{proof}
Let $x_0$ be a point in $\R^2$ such that $\frac{|D\cap B(x_0)|}{\pi} = \mathcal{A}(D)$.  By Lemma~\ref{claimford2231}, we can pick $\Omega_1$ and $\alpha_1<\frac{1}{2}$ such that if $\Omega < \Omega_1$ and $\mathcal{A}(D)<\alpha_1$, then
\begin{align}
|x_0| \le 4 \mathcal{A}(D). \label{claimford22}
\end{align}
Let us assume for a moment that the claim is true. If $\mathcal{A}(D) \ge \alpha_1$, then it follows from the definition of $\mathcal{A}_0$ that
\begin{align}\label{comparison111}
\mathcal{A}_0(D) \le \frac{|D| + |B|}{\pi} = 2 \le \frac{2}{\alpha_1}{\mathcal{A}(D)}.
\end{align}
Now let us assume that $\mathcal{A}(D)<\alpha_1$. For a constant $c>0$ such that $|B\triangle B(x_0)| \le c|x_0|,$ we can compute
\begin{align}\label{comparison112}
\mathcal{A}_0(D) = \frac{ |D \triangle B|}{\pi} & \le \frac{|D\triangle B(x_0)| + |B\triangle B(x_0)| }{\pi} \le \mathcal{A}(D) + \frac{c|x_0|}{\pi} \le \left( 1+ \frac{4 c}{\pi}\right)\mathcal{A}(D),
\end{align}
where the last inequality follows from \eqref{claimford22}.  Therefore \eqref{comparison} follows from \eqref{comparison111} and \eqref{comparison112} by choosing $c_1:=\max\left\{\frac{2}{\alpha_1},\left(1+\frac{4c}{\pi}\right)\right\}$.
\end{proof}

 In the next lemma, we will prove that if $\mathcal{A}_0(D)$ is sufficiently small, then $D$ is necessarily star-shaped. 
 \begin{lemma}\label{starshapelem}
 There exist positive constants $\Omega_2$, $\alpha_2$ and $c_2$ such that if $(D,\Omega)$ is a solution to \eqref{rotatingpatch} with $\Omega < \Omega_2$ and $\mathcal{A}_{0}(D)<\alpha_2$, then there exists $u\in C^{1}(\mathbb{T})$ such that
 \begin{align}\label{starshape}
 \partial D = \left\{ (1+u(\theta))(\cos{\theta}, \sin{\theta}) : \theta\in \mathbb{T}\right\},
 \end{align}
and
 \begin{align}\label{logbound}
 \rVert u\rVert_{L^\infty} \le c_2 \mathcal{A}_0(D)\left| \log\mathcal{A}_0(D) \right|.
 \end{align}
 \end{lemma}

\begin{proof} Without loss of generality, we assume that $D \ne B$.
The key observation is that if $\Omega$ and $\mathcal{A}_0(D)$ are sufficiently small, then the radial derivative of the relative stream function $\Psi$ is strictly positive near $\partial B$, while $\partial D$ is a connected level set of $\Psi$.

 To prove the lemma, let us consider the following decomposition of $\Psi$:
 \begin{align*}
 \Psi(x) := 1_D * \mathcal{N} -\frac{\Omega}{2}|x|^2 = \underbrace{1_B * \mathcal{N} -\frac{\Omega}{2}|x|^2}_{=:\Psi^{rad}(x)} + \underbrace{(1_D - 1_B)*\mathcal{N}(x)}_{=:\Psi^{e}(x)}.
 \end{align*}
 We claim that there exist positive constants $\Omega_1$ and $\alpha_1$ such that if $(D,\Omega)$ is a solution to \eqref{rotatingpatch} with $\Omega<\Omega_1$ and $\mathcal{A}_0(D)<\alpha_1$, then it holds for some $c$, $C>0$ that 
\begin{align}
 \partial_{r} \Psi^{rad}(x) > c \text{ for }  |x|\in \left(\frac{7}{8},\frac{9}{8}\right) \label{claimford4}\\
  |\Psi^e(x)| < C\mathcal{A}_0(D)\left| \log{\mathcal{A}_0(D)} \right| \text{ for } x\in \R^2. \label{claimford5}
\end{align}
Let us assume for a moment that \eqref{claimford4} and \eqref{claimford5} are true. Then we set 
\begin{align*}
\Omega_2:= \Omega_1 \quad \text{ and }\quad \alpha_2:=\left\{ \alpha_1, \alpha^*\right\},
\end{align*}
 where $\alpha^* = \min\left\{ \alpha > 0 : \frac{C}{c}\alpha\log\frac{1}{\alpha} = \frac{1}{16}\right\}$. If $\Omega < \Omega_2$ and $\mathcal{A}_0(D)<\alpha_2$, then  for any $x_1$ and $x_2$ such that 
\begin{align*}
x_1\in \partial D \cap \partial B, \quad \text{ and }\quad  |x_2| = 1- \frac{2C}{c}\mathcal{A}_0(D)\left| \log{\mathcal{A}_0(D)} \right|> \frac{7}{8},
\end{align*}
 we have
\begin{align}
\Psi(x_1) - \Psi(x_2) & = \left(\Psi^{rad}(x_1) - \Psi^{rad}(x_2)\right) + \left(\Psi^e(x_1) - \Psi^{e}(x_2)\right)\nonumber \\
& > c \left(|x_1| - |x_2|\right) + \left(\Psi^e(x_2) - \Psi^{e}(x_1)\right) \nonumber\\
 & > 0,\label{levelsets}
\end{align}
where the first and the second inequalities follow from \eqref{claimford4} and \eqref{claimford5} respectively. In the same way, one can easily show that for any $x_3$ such that $|x_3| = 1+ \frac{2C}{c}\mathcal{A}_0(D)\log\frac{1}{\mathcal{A}_0(D)} < \frac{9}{8}$, we have $\Psi(x_3)-\Psi(x_1) > 0$. Since $\partial D$ is a connected level set of $\Psi$ and $\partial B \cap \partial D \ne \emptyset$,  we get
\begin{align}\label{levelsets2}
\partial D \subset \left\{ x \in \R^2 : 1- \frac{2C}{c}\mathcal{A}_0(D)\log\frac{1}{\mathcal{A}_0(D)}  < |x| < 1+\frac{2C}{c} \mathcal{A}_0(D)\log\frac{1}{\mathcal{A}_0(D)} \right\}.
\end{align}
Hence  the implicit function theorem with \eqref{claimford4} and \eqref{levelsets2} yields that there exists $u\in C^{1}(\mathbb{T})$ such that \eqref{starshape} holds. Furthermore,  \eqref{levelsets2} immediately implies \eqref{logbound}.

To complete the proof, we need to prove the claims. To prove \eqref{claimford4}, note that $\partial_r \Psi^{rad}(r)$ is explicit and given by
\begin{align*}
\partial_r \Psi^{r}(r) = 
\begin{cases}
\left( \frac{1}{2}-\Omega\right) r & \text{ if }r \le 1\\
\frac{1}{2r} - \Omega r & \text{ if }r > 1.
\end{cases}
\end{align*}
Then \eqref{claimford4} follows immediately by choosing sufficiently small $\Omega_1$. For \eqref{claimford5}, note that  Lemma~\ref{claimford2231}  implies that we can choose $\Omega_1$ and $\alpha_1$ so that $ D \subset  B_3$. Then we have for any $x\in \R^2$ that
\begin{align*}
|\Psi^e(x)| & = \bigg| \int_{y\in B_3}  \left(1_D (y)- 1_B(y)\right) \log|x-y|  dy \bigg| \\
&\lesssim \bigg| \int_{y\in B_3, |x-y| < 10}   \left(1_D (y)- 1_B(y)\right) \log|x-y|  dy \bigg| +  \bigg| \int_{y\in B_3, |x-y| > 10}   \left(1_D (y)- 1_B(y)\right) \log|x-y|  dy \bigg| \\
&\lesssim \int_{y\in B_3, |x-y| < 10} |  1_D (y)- 1_B(y) | \bigg|\log|x-y| \bigg| dy +\bigg| \int_{y\in B_3, |x-y| > 10}   \left(1_D (y)- 1_B(y)\right) \frac{\log|x-y|}{\log x}  dy \bigg|\\
& \lesssim |D\triangle B | \bigg|\log {|D\triangle B|}\bigg| \\
& \lesssim \mathcal{A}_0(D)\bigg|\log{\mathcal{A}_0(D)}\bigg|,
\end{align*}
where we used $\int 1_D(y) - 1_B(y) dx = 0$ to get the second inequality. This proves \eqref{claimford5}. 
\end{proof}

The proof of the following proposition will be postponed to the next subsection. 
 \begin{prop}\label{boundomega}
 There exist  positive constants $\Omega_3$ and $\alpha_3$ such that if $ \left( D, \Omega \right)$ is a solution to \eqref{rotatingpatch} with $\Omega < \Omega_3$ and $D$ is a star-shaped domain with $\mathcal{A}_0(D) \le \alpha_3$, then $D = B$.
\end{prop}

Now we are ready to prove Theorem~\ref{smallomega}.

 \begin{proofthm}{smallomega} 
 We will choose $\Omega_0$ and $\alpha_0$ so small that all the previous lemmas are applicable. Let us set $\Omega_0:=\min\left\{ \Omega_1, \Omega_2,\Omega_3,\frac{1}{4}\right\}$ and  $\alpha_0 : = \min\left\{ \frac{\alpha_2}{c_1}, \frac{\alpha_3}{c_1}  \right\}$, where $\alpha_i's$ and $c_1$ are as in Lemma~\ref{comparisonasymmetry}, ~\ref{starshapelem} and Proposition~\ref{boundomega}. Moreover,  let $\sigma$ be as described in Proposition~\ref{alphasquared}. We assume that $(D,\Omega)$ is a solution to \eqref{rotatingpatch} with $\Omega <\Omega_0$ and $D \ne B$. Then we will prove
\begin{align}\label{equation50}
\sqrt{\frac{\sigma}{2\pi}}\alpha_0 \Omega^{-\frac{1}{2}} < \sup_{x\in \partial D}|x|.
\end{align}

  Since $D\ne B$, we have $\mathcal{A}(D)\ge \alpha_0$. Indeed, if $\mathcal{A}(D)<\alpha_0$, then Lemma~\ref{comparisonasymmetry} and Lemma~\ref{starshapelem} imply that $D$ is star-shaped and $\mathcal{A}_0(D) < c_1\alpha_0 <\alpha_3$. Therefore, Proposition~\ref{boundomega} yields that $D=B$, which is a contradiction. Thus it follows from  \eqref{variationalidentity2} and \eqref{alphasquared1} that
\begin{align*}
\Omega \int_D |x|^2 dx \ge \left( 1-2\Omega\right) \left(\frac{|D|^2}{4\pi} - \int_D pdx\right) > \frac{1}{2}\sigma\alpha_0^2,
\end{align*}
where we used $\Omega < \frac{1}{4}$.
It is clear that $\Omega\int_D |x|^2dx \le \pi\Omega\left(\sup_{x\in \partial D}|x|\right)^2$, hence the above inequality yields \eqref{equation50}.
  \end{proofthm}

 \subsection{Proof of Proposition~\ref{boundomega}}\label{proofofboundomega}
In this subsection, we aim to prove Proposition~\ref{boundomega}. We say a simply-connected bounded domain is star-shaped if there exist $u:\mathbb{T} \mapsto (-1,\infty)$ such that

\begin{align*}
\partial D = \left\{ (1+u(\theta))(\cos{\theta},\sin{\theta})\in \R^2 : \theta\in \mathbb{T} \right\}.
\end{align*}
 If $|D| = \pi$, we have that
 \begin{align*}
 \pi = \int_{\R^2} 1_D(x)dx = \int_{\mathbb{T}}\int_0^{(1+u(\theta))}rdrd\theta = \pi+ \frac{1}{2}\int_{\mathbb{T}} u(\theta)^2 + 2u(\theta)d\theta,
 \end{align*}
 thus
 \begin{align}\label{l1equall2}
 \int_{\mathbb{T}}u(\theta)^2 d\theta = - \int_{\mathbb{T}}2u(\theta) d\theta.
 \end{align}
 Furthermore $\mathcal{A}_0(D)$ and the difference of second moments of $1_Ddx$ and $1_Bdx$ can be written in terms of $u$ as
   \begin{align}
   &\mathcal{A}_0(D)  = \frac{1}{\pi}\int_{\mathbb{T}}  | u(\theta) | +sgn(u(\theta))\frac{u(\theta)^2}{2}d\theta, \label{asymmetryinu}\\
   & \int_D {|x|^2}dx - \frac{|D|^2}{2\pi} =\int_{\mathbb{T}}|u(\theta)|^2 + u(\theta)^3 +\frac{1}{4}u(\theta)^4 d\theta, \label{secondmoment}
   \end{align}
   where 
   \begin{align*}
   sgn(x) = 
   \begin{cases}
   1 & \text{ if }x\ge0\\
   -1 & \text{ otherwise}.
   \end{cases}
   \end{align*}
   Note that if $\rVert u\rVert_{L^{\infty}(\mathbb{T})} < \frac{1}{2}$, then \eqref{asymmetryinu} and \eqref{secondmoment} imply that there exists $c_3>0$ such that
   \begin{align}
   &\frac{1}{c_3}\int_{\mathbb{T}} |u(\theta)|d\theta \le \mathcal{A}_0(D) \le c_3 \int_{\mathbb{T}} |u(\theta)|d\theta, \label{asymmetryinu2}\\
   &\frac{1}{c_3}\int_{\mathbb{T}}|u(\theta)|^2d\theta \le \int_D \frac{|x|^2}{2}dx - \frac{|D|^2}{4\pi} \le c_3\int_{\mathbb{T}}|u(\theta)|^2d\theta. \label{secondmoment2}
   \end{align}
   The proof of Proposition~\ref{boundomega} is based on the identity  \eqref{variationalidentity1}. We will estimate the right-hand side of \eqref{variationalidentity1} in the following proposition.

 \begin{prop}\label{l1andl2bound}
 Let $D$ be a star-shaped domain parametrized by $u:\mathbb{T} \mapsto \R$ with $\rVert u\rVert_{L^\infty}<\frac{1}{2}$.
 Then there exists $\delta>0$ such that for any $a\in \left( 2 \rVert u \rVert _{L^{\infty}(\mathbb{T})} , 1 \right)$,  it holds that
 
 \begin{align}\label{l1andl2bound3}
\int_D |x - 2\nabla \left( 1_D * \mathcal{N} \right) |^2 dx \le \delta \left( a \int_{\mathbb{T}}|u|^2d\theta+\frac{1}{a}\int_{\mathbb{T}}f(\theta)^2d\theta \right),
 \end{align}
 where $f(\theta) := \int_0^{\theta} u(s)^2+2u(s)ds$. 
 \end{prop}
 
The above proposition will play a key role in the proofs of Proposition~\ref{boundomega} and Theorem~\ref{largem2}. In the proof of Proposition~\ref{boundomega}, we simply use  $|f(\theta)| \lesssim \rVert u \rVert_{L^{1}(\mathbb{T})}$, so that the left-hand side can be almost bounded by $L^1$-norm of $u$. Note that if we can choose $a$ small enough, then the proposition, together with \eqref{variationalidentity2} and \eqref{secondmoment2} will give  $\rVert u \rVert_{L^{2}(\mathbb{T})} \lesssim  \rVert u \rVert_{L^{1}(\mathbb{T})}$.

  In section~\ref{section4}, we will use the fact that if $u(\theta)$ is $\frac{2\pi}{m}$ periodic, then $f(\theta)$ is also $\frac{2\pi}{m}$-periodic, which follows from \eqref{l1equall2}. This will be used for the proof of Theorem~\ref{largem2}.

 \begin{proof}
 Using Cauchy-Schwarz inequality, we obtain that
 \begin{align}\label{Hterm}
 \int_D |x - 2\nabla \left( 1_D * \mathcal{N} \right) |^2 dx \lesssim \int_D | x- 2\nabla \left( 1_B * \mathcal{N}\right) |^2dx + \int_D | \nabla \mathcal{N}* \left( 1_B - 1_D \right) |^2dx =: H_1 + H_2
 \end{align} 
 To estimate $H_1$, note that
 \begin{align*}
\nabla \left( 1_B * \mathcal{N} \right) =
\begin{cases}
 \frac{x}{2} & \text{ if }|x|\le 1 \\
\frac{x}{2|x|^2} & \text{ if }|x| > 1.
\end{cases}
\end{align*}
Therefore we can compute
\begin{align*}
\int_D | x- 2\nabla \left( 1_B * \mathcal{N}\right) |^2dx&  = \int_{D\backslash B} \bigg| x - \frac{x}{|x|^2}\bigg|^2dx\\
& = \int_{D\backslash B} |x|^2 - 2 + \frac{1}{|x|^2}dx \\
&  = \int_{\mathbb{T}\cap \left\{ u > 0 \right\}} \int_1^{1+u(\theta)}.\left(r^2-2+\frac{1}{r^2}\right)rdrd\theta
\end{align*}
However, we have that for $u(\theta)>0$, 
\begin{align*}
\int_1^{1+u(\theta)} r^3-2r+\frac{1}{r}dr & = \frac{1}{4}u(\theta)^4+u(\theta)^3+\frac{1}{2}u(\theta)^2-u(\theta)+\log(1+u(\theta))\\
& \le \frac{1}{4}u(\theta)^4 + \frac{4}{3}u(\theta)^3 \\
&\lesssim \rVert u\rVert_{L^{\infty}(\mathbb{T})}|u(\theta)|^2\\
& \lesssim a|u(\theta)|^2,
\end{align*}
where we used $\log(1+x)\le x-\frac{1}{2}x^2+\frac{1}{3}x^3$ for $x \ge 0$ and $ 0\le u(\theta)< \frac{1}{2}$. Hence it follows that 
\begin{align}\label{H1estimate}
H_1 \lesssim a\int_{\mathbb{T}}|u|^2d\theta.
\end{align}

 In order to estimate $H_2$, we recall the following result:
 
 \begin{prop}\cite[Proposition 3.1]{loeper2006uniqueness} \label{estimate}
Let $\rho_1$ and $\rho_2$ be two probability measures on $\R^d$ with $L^{\infty}$ densities with respect to Lebesgue measure. Then

\begin{align*}
\rVert \nabla (\mathcal{N}*(\rho_1-\rho_2))\rVert_{L^{2}(\R^d)}^2 \le \max (\rVert \rho_1\rVert_{L^{\infty}}, \rVert \rho_2\rVert_{L^{\infty}}) W_2^2(\rho_1,\rho_2),
\end{align*}
where $W_2(\rho_1,\rho_2)$ denotes $2$-Wasserstein distance between $\rho_1$ and $\rho_2$ defined by
\begin{align*}
W^2_2(\rho_1,\rho_2) := \inf\left\{\int |T(x) - x|^2 d\rho_1(x) : T_{\#}\rho_1 = \rho_2 \right\}.
\end{align*}
 \end{prop}

  Thanks to Proposition~\ref{estimate}, it follows that 
 
 \begin{align}\label{estimate2}
 H_2 = \rVert \nabla (\mathcal{N}*(1_D-1_B))\rVert_{L^2(\R^2)}^2 \le \int_{D}| T(x)-x|^2dx,
 \end{align}
 for any $T : D \mapsto B$ such that 
 \begin{align}\label{relation}
 T_{\#}\left(1_D(x)dx\right)=1_B(x)dx,
 \end{align}
 where $T_{\#}\rho$ denotes the pushforward measure of $\rho$ by $T$.
 Note that in polar coordinates, \eqref{relation} is equivalent to
 \begin{align}\label{relation2}
 T_{\#}\left(1_{\tilde{D}}(r,\theta)rdrd\theta\right) = 1_{\tilde{B}}(r,\theta)rdrd\theta,
 \end{align}
where $\tilde{D}:=\left\{ (r,\theta)\in \left[0,1\right)\times \mathbb{T} : 0\le r < 1+u(\theta)\right\}$ and $\tilde{B} := \left\{ (r,\theta) \in \left[0,1\right)\times \mathbb{T} : 0 \le r < 1 \right\}.$ Hence it suffices to find a transport map $T$ which gives the desired estimate.

 Let us define $T : \tilde{D} \mapsto \tilde{B}$  by, 
\begin{align}\label{transportmap}
T(r,\theta) := \left( T^r(r,\theta), T^{\theta}(\theta)\right) :=
\begin{cases}
 \left( \sqrt{\frac{a(2-a)(r^2-(1+u(\theta))^2)}{(u(\theta)+a)(u(\theta)+2-a)}+1} , \frac{f(\theta)}{a(2-a)}+\theta \right) & \text{ if } r > 1-a\\
 \left( r,\theta \right) & \text{ if } r \le 1-a.
\end{cases}
\end{align}
for $a\in (2 \rVert u\rVert_{L^\infty},1)$, where $f(\theta) : = \int_0^\theta u(\eta)^2+2u(\eta)d\eta$.
 
  \begin{figure}[htbp]
\begin{center}
    \includegraphics[scale=0.6]{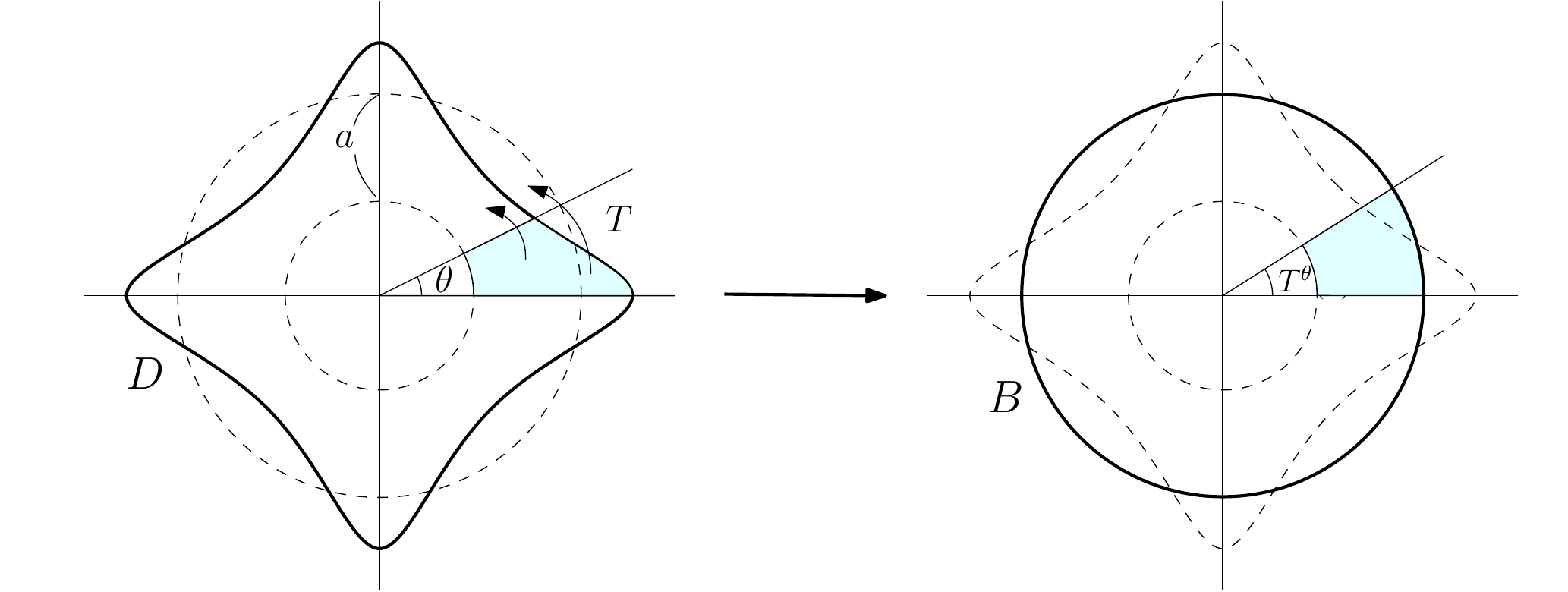}
    \caption{Illustration of the transport map $T$ that pushes forwards $D$ to $B$.}
   \label{transportmap}
    \end{center}

 Our motivation for the transport map $T$ is the following: We first choose $T^{\theta}$ so that $T^{\theta}$ is independent of $r$ and preserves the area in the sense that (see Figure~\ref{transportmap} for the illustration)
\end{figure}

\begin{align*}
\int_{0}^{\theta}\int_{1-a}^{1+u(s)}rdrds = \int_0^{T^{\theta}(\theta)}\int_{1-a}^1 rdrds.
\end{align*}
And then, we choose $T^{r}(r,\theta)$ so that \eqref{relation2} is satisfied. Note that in order to check the condition \eqref{relation2} for $T$, it suffices to show that 
\begin{align}\label{relation3}
1_{\tilde{D}}(r,\theta)r=1_{\tilde{B}}(T(r,\theta))T^r(r,\theta)|\text{det}(\nabla T)|,
\end{align}
almost everywhere with respect to the measure $1_{\tilde{D}}rdrd\theta$  (see \cite{santambrogio2015optimal}).
Then it is clear that $\theta \mapsto T^\theta(\theta)$ and $r\mapsto T^r(r,\theta)$ are increasing for fixed $r$ and $\theta$ respectively. Indeed,
\begin{align*}
\frac{d}{d\theta}T^\theta(\theta) = 1+ \frac{u(\theta)^2+2u(\theta)}{a(2-a)} \ge \frac{2a-a^2+\frac{a^2}{4}-a}{a(2-a)} \ge \frac{a-\frac{3}{4}a^2}{a(2-a)} > 0,
\end{align*}
where the first inequality follows from that $\rVert u \rVert_{L^{\infty}(\mathbb{T})} < \frac{1}{2}a$ and $x\mapsto x^2+2x$ is increasing for $x\ge -1$ thus $u(\theta)^2+2u(\theta) \ge \frac{a^2}{4}-a$.
Since $T$ maps $\left\{ (r,\theta) : r=1-a \text{ or } 1+u(\theta)\right\}$ to $\left\{ (r,\theta) : r=1-a \text{ or }r=1\right\}$ continuously, $T$ is bijective and therefore  $1_{\tilde{D}}(r,\theta)=1_{\tilde{B}}\circ T (r,\theta)$. Furthermore, the Jacobian matrix of $T$ can be computed as
\begin{align*}
\nabla T(r,\theta)= 
\begin{cases}
\begin{pmatrix}
\frac{1}{T^r(r,\theta)}\frac{a(2-a)r}{(u(\theta)+a)(u(\theta)+2-a)} & \partial_\theta T^{r}(r,\theta) \\
0 & \frac{(u(\theta)+a)(u(\theta)+2-a)}{a(2-a)}
\end{pmatrix} & \text{ if } 1-a < r < 1+u(\theta), \\
\begin{pmatrix}
1 & 0 \\
0 & 1
\end{pmatrix} & \text{ otherwise,}
\end{cases}
\end{align*}
therefore
\begin{align*}
T^{r}(r,\theta)|\text{det}(\nabla T)|=r,
\end{align*}
almost everywhere.
This implies that $T$ satisfies \eqref{relation3} and thus \eqref{relation2} holds. Then it follows from \eqref{estimate2} that
\begin{align*}
 H_2 \le \int_{\mathbb{T}}\int_{1-a}^{1+u(\theta)}|T^r(r,\theta) \cos(T^\theta) - r\cos\theta|^2 + |T^r(r,\theta) \sin(T^\theta) - r\sin\theta|^2 rdrd\theta.
\end{align*}
The cosine term in the integrand can be estimated as
\begin{align*}
|T^r(r,\theta)\cos(T^\theta)-r\cos\theta|^2 &= |(T^{r}(r,\theta)-r)\cos(T^\theta(\theta))+r(\cos(T^\theta(\theta))-\cos\theta)|^2\\
& \lesssim |T^r(r,\theta)-r|^2 + |\cos(T^\theta(\theta))-\cos\theta|^2\\
& \lesssim |T^r(r,\theta)-r|^2 + |T^\theta(\theta)-\theta|^2.
\end{align*}
In the same way, the sine term can be bounded as $|T^r\sin(T^\theta)-r\sin\theta|^2 \lesssim |T^r-r|^2 + |T^\theta-\theta|^2$, thus we have

\begin{align}\label{abound}
H_2 \lesssim \int_{\mathbb{T}}\int_{1-a}^{1+u(\theta)}|T^r(r,\theta)-r|^2 rdrd\theta + \int_{\mathbb{T}}\int_{1-a}^{1+u(\theta)}|T^\theta(\theta)-\theta|^2 rdrd\theta =: A_1 + A_2.
\end{align}
$A_2$ is bounded by
 
\begin{align}\label{a2bound}
A_2 \le \int_{\mathbb{T}}\int_{1-a}^{1+u(\theta)}\frac{f(\theta)^2}{a^2}rdrd\theta \lesssim \int_{\mathbb{T}}f(\theta)^2\frac{|u(\theta)|+a}{a^2}\lesssim \frac{1}{a}\int_{\mathbb{T}}f(\theta)^2d\theta,
\end{align}
where we used $\rVert u\rVert_{L^{\infty}(\mathbb{T})} < a$ to get the first and the last inequalities.

For $A_1$, we assume for a momoent that for $r \in (1-a, 1+u(\theta))$,
\begin{align}\label{estimatetr}
| T^{r}(r,\theta)-r | \lesssim | u(\theta) |.
\end{align}

From \eqref{estimatetr}, we obtain
\begin{align}\label{a1bound}
A_1 \lesssim \int_{\mathbb{T}}\int_{1-a}^{1+u(\theta)}|u(\theta)|^2rdrd\theta=\int_{\mathbb{T}}|u(\theta)|^2\int_{1-a}^{1+u(\theta)}rdrd\theta\lesssim a\int_{\mathbb{T}}|u(\theta)|^2d\theta,
\end{align}
where the last inequality follows from $a>\rVert u\rVert_{L^{\infty}(\mathbb{T})}.$
Therefore, it follows from \eqref{abound}, \eqref{a2bound} and \eqref{a1bound} that
\begin{align}\label{H2bound}
H_2 \lesssim a\int_{\mathbb{T}}|u|^2d\theta +\frac{1}{a}\int_\mathbb{T} |f|^2d\theta.
\end{align}
Thus \eqref{l1andl2bound3} follows from \eqref{Hterm}, \eqref{H1estimate} and \eqref{H2bound}. 

 To check \eqref{estimatetr}, let $g(a,r,x):= \frac{\sqrt{\frac{a(2-a)(r^2-(1+x)^2)}{(x+a)(x+2-a)}+1}-r}{x}$ so that $\frac{T^{r}(r,\theta)-r}{u(\theta)} = g(a,r,u(\theta))$. Then it suffices to show that $|g(a,r,x)| \lesssim 1$ in $\left\{(a,r,x) : (1-a)<r<1+x,\ 2|x|<a<1\right\}$. Since $g(a,r,x)$ is continuous everywhere except for $x=0$, we only need to check  $|g(a,r,x)| \lesssim 1$ when $0 < x\ll 1$. Taking the limit, we obtain
\begin{align*}
\lim_{x\to 0^+}g(a,r,x) = \frac{\frac{\partial }{\partial x}\left(\sqrt{\frac{a(2-a)(r^2-(1+x)^2)}{(x+a)(x+2-a)}+1}-r\right)\bigg|_{x=0}}{1} = \frac{(1-r^2)-a(2-a)}{ra(2-a)}.
\end{align*}
If $r<\frac{1}{2}$, then $a>\frac{1}{2}$ therefore it follows from $r>1-a>0$ that
\begin{align*}
\lim_{x\to 0}|g(a,r,x)| = \frac{|r^2-(a-1)^2|}{ra(2-a)}\le \frac{r}{a(2-a)}+\frac{(a-1)^2}{ra(2-a)} \le \frac{2r}{a(2-a)}\lesssim 1,
\end{align*}
where the second inequality follows from $(1-a) < r$.
If $r>\frac{1}{2}$, then it follows from $|r-1| < a$ that
\begin{align*}
\lim_{x\to 0}|g(a,r,x)| \le \frac{|1-r^2|}{ra(2-a)} + \frac{(2-a)}{r(2-a)} < \frac{1+r}{r(2-a)} + \frac{2-a}{r(2-a)} \lesssim 1.
\end{align*}
This proves \eqref{estimatetr} and finishes the proof.
 \end{proof}

Now we are ready to prove Proposition~\ref{boundomega}.

\begin{proofprop}{boundomega}
We will fix $\Omega_3$ and $\alpha_3$ so small that all the lemmas are applicable. To do so, let us denote $h(x):=-x\log x$. Also we denote by $\alpha^*>0$ the smallest positive number such that 
\begin{align}\label{alphastar}
\left(h(\alpha^*)\right)^3 > (\alpha^*)^2.
\end{align}
Furthermore, let $\Omega_i's$, $\alpha_i's$ and $c_i's$ for $i=1,2$ be as in Lemma~\ref{comparisonasymmetry} and Lemma~\ref{starshapelem}, let $c_3$ be as in \eqref{asymmetryinu2} and \eqref{secondmoment2} and let $\delta$ be as in Proposition~\ref{l1andl2bound}. Lastly, let $c_4:=  18\pi c_3^3\delta^2$. Then let us fix
\begin{align}\label{parameters3}
\Omega_3:=\min\left\{\Omega_1, \Omega_2,\frac{1}{4}, \frac{\sigma}{16c_1^2c_3^2c_4} \right\} \quad and  \quad \alpha_3 :=\left\{\alpha^*,\alpha_1,\alpha_2,\left(\frac{1}{2c_2}\right)^{\frac{3}{2}}, \left(\frac{1}{4c_2c_3\delta}\right)^{\frac{3}{2}}\right\}.
\end{align}
 Then our goal is to show that if $(D,\Omega)$ is a solution to \eqref{rotatingpatch} with $\Omega < \Omega_3$ and $\mathcal{A}_{0}(D) < \alpha_3$, then $D=B$.

\textbf{Step 1.} Let us claim that
\begin{align}
& \mathcal{A}_0(D) \le c_1 \mathcal{A}(D). \label{asymcosmparison}\\
&\rVert u \rVert_{L^{\infty}(\mathbb{T})} \le c_2 \mathcal{A}_0(D)^{\frac{2}{3}} \le \frac{1}{2}.\label{ulinfty}
\end{align}
Since $\Omega_3< \Omega_1$ and $\mathcal{A}(D) \le \mathcal{A}_0(D)< \alpha_3 \le \alpha_1$, it follows from Lemma~\ref{comparisonasymmetry} that $\mathcal{A}_0(D) < c_1 \mathcal{A}(D)$. In addition, $\Omega_3 < \Omega_2$, $\mathcal{A}_0(D) < \alpha_3\le \alpha_2$ and Lemma~\ref{starshapelem} imply that
\begin{align*}
\rVert u\rVert_{L^{\infty}(\mathbb{T})} \le c_2h(\mathcal{A}_0(D)) \le c_2 \mathcal{A}_0(D)^{\frac{2}{3}},
\end{align*}
where the last inequality follows from $\alpha_3\le \alpha^*$. Since $\mathcal{A}_0(D)<\alpha_3\le \left(\frac{1}{2c_2}\right)^{\frac{3}{2}}$, we have $c_2\mathcal{A}_0(D)^{\frac{2}{3}} \le \frac{1}{2}$, which proves \eqref{ulinfty}.

\textbf{Step 2.} In this step, we will show that
\begin{align}\label{step2claim}
\frac{1}{2}\int_D |x - 2\nabla \left( 1_D * \mathcal{N} \right) |^2 dx \le \frac{1}{4c_3}\int_{\mathbb{T}} |u|^2d\theta + c_4\mathcal{A}_0(D)^2,
\end{align}
where $c_4:= 18\pi c_3^3\delta^2$.
Since $\rVert u \rVert_{L^{\infty}(\mathbb(T))} < \frac{1}{2}$, we will apply Proposition~\ref{l1andl2bound} with $a:= \frac{1}{2c_3\delta}$. Note that 
\begin{align*}
2\rVert u \rVert_{L^{\infty}(\mathbb{T})} \le 2c_2\mathcal{A}_0(D)^{\frac{2}{3}} <2c_2\alpha_3^{\frac{2}{3}} \le a, 
\end{align*}
where the first inequality follows from \eqref{ulinfty}, the second follows from the assumption that $\mathcal{A}_0(D)<\alpha_3$ and the last inequality follows from \eqref{parameters3}, which says $\alpha_3\le \left(\frac{1}{4c_2c_3\delta}\right)^{\frac{3}{2}}$.
Thus we can obtain by using Proposition~\ref{l1andl2bound} that 
\begin{align}\label{equation32}
\frac{1}{2}\int_D |x - 2\nabla \left( 1_D * \mathcal{N} \right) |^2 dx \le \frac{1}{4c_3}\int_{\mathbb{T}}|u|^2 d\theta + c_3\delta^2 \int_{\mathbb{T}}f(\theta)^2d\theta,
\end{align}
where $f(\theta) = \int_0^\theta u(s)^2 + 2u(s) ds $. Moreover, we have 
\begin{align}\label{equation57}
|f(\theta)| \le \int_{0}^{\theta}3|u(s)|ds < \int_{\mathbb{T}} 3|u(s)|ds \le 3c_3 \mathcal{A}_0(D),
\end{align}
where the last inequality follows from \eqref{asymmetryinu2}.  Therefore it follows from \eqref{equation32} and \eqref{equation57} that
\begin{align*}
\frac{1}{2}\int_D |x - 2\nabla \left( 1_D * \mathcal{N} \right) |^2 dx \le \frac{1}{4c_3}\int_{\mathbb{T}} |u|^2d\theta + 18\pi c_3^3\delta^2\mathcal{A}_0(D)^2,
\end{align*}
which proves the claim \eqref{step2claim}.

\textbf{Step 3.} Now we will prove that
\begin{align}\label{l2lessl1}
\int_{\mathbb{T}} |u(\theta)|^2d\theta \le 4c_3c_4\mathcal{A}_0(D)^2.
\end{align}
 Since $\Omega < \Omega_3 \le \frac{1}{4}$, it follows from \eqref{secondmoment2} that
 \begin{align}
 \left( 1-2\Omega\right)\left(\int_{D}\frac{|x|^2}{2}dx-\frac{|D|^2}{4\pi} \right)>\frac{1}{2c_3}\int_{\mathbb{T}}|u(\theta)|^2d\theta.
 \end{align}

Thus it follows from \eqref{variationalidentity1} and \eqref{step2claim} that
\begin{align*}
\frac{1}{2c_3}\int_{\mathbb{T}}|u(\theta)|^2d\theta < \frac{1}{4c_3}\int_{\mathbb{T}}|u(\theta)|^2d\theta + c_4\mathcal{A}_0(D)^2,
\end{align*}
which proves \eqref{l2lessl1}.

\textbf{Step 4.} Finally, we will prove $D=B$ by showing that $\mathcal{A}_0(D)=0$. This will be done by estimating the left/right-hand side in \eqref{variationalidentity2}. It follows from Proposition~\ref{alphasquared} and \eqref{asymcosmparison} that
\begin{align}\label{lefthandside2}
\left(1-2\Omega\right)\left(\frac{|D|^2}{4\pi} - \int_D pdx\right) \ge \frac{1}{2}\sigma \mathcal{A}(D)^2 \ge \frac{1}{2c_1^2}\sigma\mathcal{A}_0(D)^2,
\end{align}
where we used $\Omega < \frac{1}{4}$. Moreover, it follows from \eqref{secondmoment2} and \eqref{l2lessl1} that
\begin{align}\label{righthandside}
2\Omega\left(\int_D \frac{|x|^2}{2}dx - \frac{|D|^2}{4\pi}\right) \le 2\Omega c_3 \int_{\mathbb{T}}|u(\theta)|^2d\theta \le 8\Omega c_3^2c_4\mathcal{A}_0(D)^2.
\end{align}
Therefore \eqref{variationalidentity2} yields that
\begin{align*}
\left(8\Omega c_3^2 c_4 - \frac{\sigma}{2c_1^2}\right)\mathcal{A}_0(D)^2 \ge 0.
\end{align*}
 This implies $\mathcal{A}_0(D) = 0$, since $8\Omega c_3^2c_4 - \frac{\sigma}{2c_1^2} < 8\Omega_3 c_3^2c_4 - \frac{\sigma}{2c_1^2} \le 0$, which follows from \eqref{parameters3} and $\Omega < \Omega_3$. This proves that $D= B$.

\end{proofprop}

\section{Rotating patches with $m$-fold symmetry}\label{section4}
We now move on to the quantitative estimates for $m$-fold symmetric rotating patches. We say a domain $D$ is $m$-fold symmetric, if $D$ is invariant under rotation by $\frac{2\pi}{m}$. We divide this section into two subsections: The first subsection is devoted to the proof of Theorem~\ref{largem2} and the second subsection is devoted to the proof of Theorem~\ref{largem}.


\subsection{Proof of Theorem~\ref{largem2}}\label{proofoftheorem2}
The goal of this subsection is to prove Theorem~\ref{largem2}. As explained in Remark~\ref{remark}, angular velocity $\Omega$ is independent of radial dilation, thus we will assume that $|D| = |B| = \pi$ throughout this subsection. 

 For a simply-connected and $m$-fold symmetric patch $D$, we denote $r_{min}:=\inf_{x\in\partial D}|x|$, and $r_{max}:=\sup_{x\in\partial D}|x|$. Note that the origin is necessarily contained in $D$ since $D$ is simply-connected and $m$-fold symmetric, therefore $r_{min}>0$. Furthermore, since we are assuming $|D| = \pi$, it is necessarily $r_{min}<1$ and $r_{max}>1$ if $D$ is not a disk. 
 
 We will prove the theorem by contrapositive. We suppose to the contrary that $(D,\Omega)$ is an $m$-fold symmetric  solution with sufficiently large $m$ and $\lambda := \frac{1}{2}-\Omega$ is sufficient large compared to $\frac{1}{m}$. Then Lemma~\ref{phimbound} tells us that the patch is necessarily star-shaped and the polar graph that parametrizes $\partial D$ must be small. With this fact, we will apply the identity \eqref{variationalidentity1} and Proposition~\ref{l1andl2bound} to derive an upper bound of $\lambda$, which we expect to contradict our initial assumption on $\lambda$.

 Now we introduce a decomposition of the stream function $1_D * \mathcal{N}$. We define a radial function $g:\R^2 \mapsto \R$ as follows (where we denote it by $g(r)$ by slight abuse of nontation): 
\begin{align*}
g(r) := \frac{1}{2\pi r}\mathcal{H}^1\left( \partial B_r \cap D\right),
\end{align*}
where $\mathcal{H}^{1}$ denotes the $1$-dimensional Hausdorff measure.
Then we shall write, in polar coordinates,
\begin{align}\label{decomp}
\left(1_D * \mathcal{N}\right)(r,\theta) = \left(g * \mathcal{N}\right)(r) + \left( 1_D - g \right) * \mathcal{N}(r,\theta) =: \varphi^{r}(r) + \varphi_{m}(r,\theta).
\end{align}
Therefore  the relative stream function can be written as $\Psi(r,\theta) = \varphi^r(r)-\frac{\Omega}{2}r^2 + \varphi_m(r,\theta)$. 

Note that $g$ is a radial function with the same integral as $1_D$ on each $\partial B_r$. If $D$ is $m$-fold symmetric for large $m$, we would expect that the velocity field generated by the vorticity $1_D$ must be very close to the velocity field generated by $g$, that is, we expect that $|\nabla \varphi_m | \ll 1$ if $m\gg1$. Below we will give a quantitative proof of this fact in Lemma~\ref{derivatives}.

\begin{lemma}\label{derivatives}
 Let $D$ be an $m$-fold symmetric bounded domain for $m\ge 3$. Then
\begin{align} 
\partial_r \varphi^{r}(r) = \frac{|D\cap B_r|}{2\pi r}, \label{radialderivative} \\
|\nabla \varphi_m(r,\theta)| \lesssim \frac{r}{m}. \label{errorderivative}
\end{align}
\end{lemma}

\begin{proof} Let us prove \eqref{radialderivative} first. Obviously, \eqref{radialderivative} is equivalent to
\begin{align}\label{equation35}
2\pi r\partial_r \varphi^r(r) = |D\cap B_r|.
\end{align}
Clearly both sides of \eqref{equation35} are zero at $r=0$. Also we have that
\begin{align*}
\partial_r \left(|D \cap B_r| \right) &= \mathcal{H}^1\left( D\cap \partial B_r\right) = 2\pi r g(r) = 2\pi r \Delta \left(\varphi^r(r) \right) = \partial_r \left( 2\pi r\partial_r \varphi^r(r)\right),
\end{align*}
where we used $\Delta = \frac{1}{r}\partial_r \left(r\partial_r\right) + \frac{1}{r^2}\partial_{\theta\theta}$. This proves \eqref{equation35}, thus \eqref{radialderivative}.

We will prove \eqref{errorderivative} by using the formula for the stream function given in Lemma~\ref{derivatives335}. Let $h(r,\theta) := 1_D(r\cos\theta,r\sin\theta)-g(r)$. We apply \eqref{formula4} and \eqref{formula3} to \eqref{errorderivative111} and \eqref{errorderivative112} respectively, and obtain

\begin{align}
\partial_r\varphi_m(r,\theta) & = \frac{1}{2\pi}\int_{\mathbb{T}}\int_0^r  h(\rho,\eta+\theta)\left(\frac{\left(\frac{\rho}{r}\right)^{m+1}\left(\cos(m\eta)-\left(\frac{\rho}{r}\right)^m\right)}{\left(1-\left(\frac{\rho}{r}\right)^m\right)^2+2\left(\frac{\rho}{r}\right)^m(1-\cos{(m\eta)})} \right)d\rho d\eta \nonumber\\
& \  - \frac{1}{2\pi}\int_{\mathbb{T}}\int_r^{\infty}  h(\rho,\eta+\theta)\left(\frac{\left(\frac{r}{\rho}\right)^{m-1}(\cos(m\eta)-\left(\frac{r}{\rho}\right)^m)}{\left(1-\left(\frac{r}{\rho}\right)^m\right)^2+2\left(\frac{r}{\rho}\right)^m(1-\cos{(m\eta)})} \right)d\rho d\eta \nonumber \\
&=: A_1 - A_2,  \label{radial1123} \\
\partial_{\theta} \varphi_m(r,\theta) &= -r\frac{1}{2\pi}\int_{\mathbb{T}}\int_0^{r}h(\rho,\eta+\theta)\left(\frac{\left(\frac{\rho}{r}\right)^{m+1}\sin(m\eta)}{\left(1-\left(\frac{\rho}{r}\right)^{m}\right)^2+2\left(\frac{\rho}{r}\right)^m(1-\cos(m\eta))}\right)d\rho d\eta \nonumber\\
&\ - r\frac{1}{2\pi}\int_{\mathbb{T}}\int_r^{\infty}h(\rho,\eta+\theta)\left(\frac{\left(\frac{r}{\rho}\right)^{m-1}\sin(m\eta)}{\left(1-\left(\frac{r}{\rho}\right)^{m}\right)^2+2\left(\frac{r}{\rho}\right)^m(1-\cos(m\eta))}\right)d\rho d\eta \nonumber \\
& =: -r A_3 -r A_4\label{angular1123}
\end{align}
We claim that
\begin{align}\label{claim442}
|A_i| \lesssim \frac{r}{m} \text{ for  }i=1,2,3 \ \text{ and }\ 4. 
\end{align}
Let us assume for a moment that the claim is true. Then \eqref{radial1123} and \eqref{angular1123} yield that $|\nabla \varphi_m(r,\theta)| \sim |\partial_{r}\varphi_{m}(r,\theta)| + |\frac{\partial_{\theta}\varphi_m(r,\theta)}{r}| \lesssim \frac{r}{m}$, which finishes the proof. We give a proof  of \eqref{claim442} for only $A_2$ since the other terms can be proved in the same way. Note that in the proof, we will see that the assumption $m\ge 3$ is crucial to estimate $A_2$ and $A_4$.  

 From the change of the variables, $\left(\frac{r}{\rho}\right)^{m} \mapsto x$ and $\frac{2\pi}{m}$-periodicity of the integrand in the angular variable, it follows that
\begin{align*}
|A_2| &\le \int_{\mathbb{T}}\int_0^{1} \bigg| h(rx^{-\frac{1}{m}},\eta+\theta) \left( \frac{x^{1-\frac{1}{m}}(\cos(m\eta)-x)}{\left( 1-x\right)^2+2x(1-\cos(m\eta))}\right)\frac{r}{m}x^{-1-\frac{1}{m}}\bigg|dxd\eta\\
& \le \frac{r}{m} \int_{\mathbb{T}}\int_{0}^{1}\bigg| \frac{x^{-\frac{2}{m}}(\cos(m\eta)-x)}{(1-x)^2+2x(1-\cos(m\eta))}\bigg|dxd\eta\\
& \le r\int_{0}^{\frac{2\pi}{m}}\int_0^{1} \frac{x^{-\frac{2}{m}}((1-x)+(1-\cos(m\eta)))}{(1-x)^2+2x(1-\cos(m\eta))}dxd\eta \\
& = \frac{r}{m}\int_{\mathbb{T}}\int_0^{1}\frac{x^{-\frac{2}{m}}((1-x) + (1-\cos \eta ))}{(1-x)^2+2x(1-\cos\eta)} dxd\eta\\
& = \frac{2r}{m}\int_{0}^{\pi} \int_0^{1}\frac{x^{-\frac{2}{m}}((1-x) + (1-\cos\eta))}{(1-x)^2+2x(1-\cos\eta)}dxd\eta\\
& = \frac{2r}{m}\left( \int_{0}^{\pi}\int_{0}^{\frac{1}{2}}\frac{x^{-\frac{2}{m}}((1-x) + (1-\cos\eta))}{(1-x)^2+2x(1-\cos\eta)} dxd\eta + \int_{0}^{\pi}\int_{\frac{1}{2}}^{1}\frac{x^{-\frac{2}{m}}((1-x) + (1-\cos\eta))}{(1-x)^2+2x(1-\cos\eta)} dxd\eta \right)\\
& = \frac{2r}{m}\left(  A_{21} + A_{22} \right)
\end{align*}
where we used $\frac{2\pi}{m}$-periodicity of the integrand to get the third inequality, the change of variables, $\eta \mapsto \frac{1}{m}\eta$ to get the first equality, and the evenness of the integrand in $\eta$ to get the second equality. Note that the denominator of the integrand $A_{21}$ is bounded from below by a strictly positive number, therefore
\begin{align*}
A_{21} \lesssim \int_{0}^{\pi}\int_{0}^{\frac{1}{2}}x^{-\frac{2}{m}}dxd\eta \lesssim \frac{m}{m-2} \lesssim 1,
\end{align*}
for $m\ge 3$.
For $A_{22}$, we use that $(1-\cos\eta) ~\sim \eta^2$ for $\eta \in (0,\pi)$ and the change of variables, $x\mapsto 1-x$, to obtain
\begin{align*}
A_{22}\lesssim \int_{0}^{\pi}\int_{0}^{\frac{1}{2}}\frac{x+\eta^2}{x^2+\eta^2}dxd\eta & = \int_{0}^{\pi}\int_{0}^{\frac{1}{2}}1_{\left\{ x<\eta\right\}} \frac{x+\eta^2}{x^2+\eta^2}dxd\eta + \int_{0}^{\pi}\int_{0}^{\frac{1}{2}}1_{\left\{ x\ge \eta\right\}} \frac{x+\eta^2}{x^2+\eta^2}dxd\eta\\
& \le \int_{0}^{\pi}\int_0^{\eta} \frac{\eta+\eta^2}{\eta^2}dxd\eta + \int_{0}^{\frac{1}{2}}\int_0^{x}\frac{x+x^2}{x^2}d\eta dx\\
&  \lesssim 1.
\end{align*}
This proves $|A_2| \lesssim \frac{r}{m}$. As mentioned, the same argument applies to $A_1$, $A_3$ and $A_4$ to prove \eqref{claim442}. This completes the proof. 

\end{proof}

 From \eqref{radialderivative} and \eqref{errorderivative} in the above lemma, it is clear that $\partial_r \Psi(r_{min},\theta) = r_{min}\left(\frac{1}{2} - \Omega - \frac{\partial_r \varphi_m(r_{min},\theta)}{r_{min}}\right)$ and $\bigg|\frac{\partial_r\varphi_m(r_{min},\theta)}{r_{min}}\bigg| \sim  \frac{1}{m}$. Thus one can expect that if $\frac{1}{2}-\Omega$ is sufficiently large  compared to $\frac{1}{m}$, then the level set $\partial D$ cannot be too far from the a circle. We give a detailed proof for this in the following lemma.

 \begin{lemma}\label{phimbound}
Assume that $(D,\Omega)$ is a solution to \eqref{rotatingpatch}. Then there exist constants $c_1,c_2>0 $ and $m_1\ge 3$ such that if $D$ is $m$-fold symmetric for some $m\ge m_1$ and $\lambda =\frac{1}{2}-\Omega > \frac{c_1}{m}$, then $D$ is star-shaped and $|r_{max}-r_{min}|< \frac{c_2}{m}$. Hence there exist $u\in C^{1}(\mathbb{T})$ such that
\begin{align*}
\partial D = \left\{ (1+u(\theta))(\cos\theta,\sin\theta) : \theta\in \mathbb{T}\right\} \quad \text{ and }  \quad \rVert u \rVert_{L^{\infty}(\mathbb{T})}< \frac{c_2}{m}.
\end{align*}

 \end{lemma}
 
 \begin{proof}
Thanks to \eqref{errorderivative} in Lemma~\ref{derivatives}, we can find a constant $C>0$ (which we can also assume to be larger then 1) such that
 \begin{align}\label{boundforphim}
 |\nabla_{x} \varphi_m(r,\theta)| < C \frac{r}{m},
 \end{align}
where $\nabla_x$ denotes the gradient in Cartesian coordinates, that is, $\nabla_x := \partial_r +\frac{1}{r}\partial_\theta$. 
 We will first prove the bound for $r_{max}-r_{min}$ and show star-shapeness of $\partial D$ afterwards. Let 
 \begin{align}\label{numbers}
 c_1:=\max\left\{ 6C, \sqrt{48\pi C}\right\}+1,\quad c_2:=\frac{c_1}{4}\quad \text{ and  } \quad m_1:= \max\left\{ \frac{3C}{2},\frac{c_1}{4},3\right\}+1.
 \end{align}
  We will show that if $\lambda  > \frac{c_1}{m}$ and $m\ge m_1$, then  $r_{max}-r_{min} < \frac{c_2}{m}$.

Let $q(r) := \frac{|D\cap B_r|}{2\pi r^2}-\Omega$. Since $\frac{1}{r^2} > \frac{1}{r_{min}^2} - \frac{2}{r_{min}^3}(r- r_{min})$ for $r>r_{min}$, and $|D\cap B_r|$ is increasing in $r$, we have that
 \begin{align*}
 q(r) >\frac{1}{2}\lambda, \text{ for }  r \in \left(r_{min}, r_{min}\left(1+\frac{1}{4}\lambda\right)\right),
 \end{align*}
 which implies that
 \begin{align}\label{radialderivative1}
 \partial_r \left(\varphi^{r}(r) - \frac{\Omega}{2}r^2\right) = rq(r) > \frac{r_{min}}{2}\lambda, \text{ for }r\in \left(r_{min}, r_{min}\left(1+\frac{1}{4}\lambda\right)\right),
 \end{align}
  where the  equality follows from \eqref{radialderivative} in Lemma~\ref{derivatives}.
Let $\ep:=\frac{r_{min}c_1}{4m}$. By the assumption $\lambda > \frac{c_1}{m}$, we have
\begin{align}\label{epsilon}
\ep<\frac{r_{min}}{4}\lambda.
\end{align}
 We choose $x_1,x_2 \in\R^2$ such that for some $\theta_1,\theta_2 \in \mathbb{T}$,
\begin{align*}
x_1 = r_{min}(\cos(\theta_1),\sin(\theta_1)),\quad x_2=(r_{min}+\ep)(\cos(\theta_2),\sin(\theta_2)) \quad \text{ and }\quad |\theta_1-\theta_2|\le \frac{2\pi}{m}.
\end{align*}
  We claim that
 \begin{align}\label{claim332}
 \Psi(x_2) - \Psi(x_1) > 0.
 \end{align} 
 Let us assume that the claim is true for a moment. Then from $m$-fold symmetry of $D$ and the fact that $\partial D$ is a level set of $\Psi$, it follows that $r_{max}\le r_{min}+\ep$. Thus it follows from  \eqref{numbers}, \eqref{epsilon} and $r_{min}<1$ that
 
 \begin{align}\label{result1}
 r_{max}-r_{min} \le \ep =\frac{c_2r_{min}}{m}<\frac{c_2}{m}.
 \end{align}
 Furthermore, for all $x \in \partial D$, it follows from \eqref{boundforphim} and \eqref{radialderivative1} that $ \partial_r \Psi(x) > \frac{r_{min}}{2}\lambda - \frac{Cr_{max}}{m}$. Hence
 \begin{align*}
 \partial_r \Psi(x) > \frac{r_{min}c_1}{2m} - \frac{Cr_{max}}{m} \ge \frac{c_{1}}{2m}(r_{max}-\frac{c_2}{m})-\frac{Cr_{max}}{m} =\left(\frac{c_1r_{max}}{4m}-\frac{Cr_{max}}{m}\right)+\frac{c_1}{2m}\left(\frac{r_{max}}{2}-\frac{c_2}{m}\right)>0,
 \end{align*}
 where the first inequality follows from $\lambda>\frac{c_1}{m}$, the second inequality follows from \eqref{result1} and the last inequality follows from \eqref{numbers} and $r_{max}\ge1$, which say $\frac{c_{1}}{4}>C$ and $\frac{r_{max}}{2}>\frac{1}{2}>\frac{c_2}{m}$. Therefore the implicit function theorem yields that there exists $u\in C^{1}(\mathbb{T})$ such that $\partial D = \left\{ (1+u(\theta))(\cos\theta,\sin\theta) : \theta\in \mathbb{T}\right\}$.  
 This proved star-shapeness of $D$ and the desired $L^{\infty}$-norm  bound for $u$.

 Now it suffices to prove \eqref{claim332}. We compute

 \begin{align*}
 \Psi(x_2) - \Psi(x_1) & = \underbrace{\left(\varphi^r(|x_2|)-\frac{\Omega}{2}|x_2|^2\right) - \left( \varphi^r(|x_1|)-\frac{\Omega}{2}|x_1|^2\right)}_{=:L_1} + \underbrace{\varphi_m(x_2) - \varphi_m(x_1)}_{L_2}.
 \end{align*}
 Thanks to \eqref{radialderivative1}, we have 
  \begin{align}\label{radialpart11}
L_1  > \frac{r_{min}}{2}\lambda \left(|x_2| - |x_1|\right) = \frac{r_{min}\lambda\ep}{2}.
 \end{align}
 To estimate $L_2$, let us pick $x_1'=r_{min}(\cos(\theta_2),\sin(\theta_2)) $. Then it follows from \eqref{boundforphim} that
 \begin{align}\label{errorpart11}
 L_2 & = \left(\varphi_m(x_2)-\varphi_m(x_1')\right)+\left(\varphi_m(x_1')-\varphi_m(x_1)\right) \nonumber\\
 & > - C\frac{|x_2|}{m} (|x_2|-|x_1|) -C\frac{r_{min}^2}{m}\frac{2\pi}{m}\nonumber\\
 & = -\frac{Cr_{min}\ep}{m} - \frac{C\ep^2}{m} -\frac{2\pi Cr_{min}^2}{m^2}.
 \end{align}
  Hence it follows from \eqref{radialpart11} and \eqref{errorpart11} that (we split $\frac{r_{min}\lambda \epsilon}{2}$ into three pieces evenly) 
 \begin{align}\label{iose}
  \Psi(x_2) - \Psi(x_1) & > \left(\frac{r_{min}\lambda\ep}{6} -\frac{Cr_{min}\ep}{m}\right) +\left(\frac{r_{min}\lambda\ep}{6} - \frac{C\ep^2}{m}\right) + \left(\frac{r_{min}\lambda\ep}{6} -\frac{2\pi Cr_{min}^2}{m^2}\right) =: L_3 + L_4 + L_5.
 \end{align}
 From $\lambda > \frac{c_1}{m}$ and \eqref{numbers}, which says $c_1\ge 6C$, we have $L_3 = \frac{r_{min}\ep}{6}\left(\lambda - \frac{6C}{m}\right) \ge 0.$ For $L_4$, it follows from that
  \begin{align*}
 L_4 = \ep\left(\frac{r_{min}\lambda}{6}-\frac{C\ep}{m}\right) > \ep\left( \frac{r_{min}\lambda}{6}-\frac{Cr_{min}\lambda}{4m}\right) = \ep\frac{r_{min}\lambda}{6}\left(1-\frac{3C}{2m}\right) > 0,
 \end{align*}
 where the first inequality follows from \eqref{epsilon} and the last inequality follows from \eqref{numbers}, which says $m\ge m_1>\frac{3C}{2}$.
 Finally, 
 \begin{align*}
 L_5 = \frac{r_{min}^2c_1\lambda}{24m}-\frac{2\pi Cr_{min}^2}{m^2} =\frac{r_{min}^2}{24m}\left( {c_1\lambda}-\frac{48\pi C}{m}\right) > \frac{r_{min}^2}{24m}\left(\frac{c_1^2}{m}-\frac{48\pi C}{m}\right)>0,
 \end{align*}
 where the first equality follows from the definition of $\ep$, the first inequality follows from $\lambda > \frac{c_1}{m}$ and the last inequality follows from \eqref{numbers}, which says $c_1 \ge \sqrt{48\pi C}$. Therefore it follows from \eqref{iose} that
 \begin{align}\label{separation}
 \Psi(x_2) - \Psi(x_1) > 0,
 \end{align}
which finishes the proof.
 \end{proof}

Now we are ready to prove Theorem~\ref{largem2}. 
\begin{proofthm}{largem2} Let $c_1$, $c_2$ and $m_1$ be constants in Lemma~\ref{phimbound} and $\delta$ be as in Proposition~\ref{l1andl2bound}. Lastly, let $c_3$ be the constant in \eqref{secondmoment2}. Now we set
\begin{align}\label{numbers2}
c:= \max\left\{c_1, \frac{c_3}{2}\left(c_2\delta+\frac{9\pi^2\delta}{c_2}\right) \right\} \quad \text{ and }\quad m_0 := \max\left\{2c_2, m_1  \right\}+1.
\end{align}
We will prove that if $(D,\Omega)$ is a solution to \eqref{rotatingpatch} such that $D$ is $m$-fold symmetric for $m\ge m_0$ and simply-connected, then 
\begin{align}\label{result2}
\lambda:=\frac{1}{2}-\Omega \le \frac{c}{m}.
\end{align} 
Towards a contradiction, let us suppose that there exists $(D,\Omega)$ such that 
\begin{align}\label{contra1}
\lambda>\frac{c}{m}.
\end{align}
 It is clear that \eqref{numbers2} implies $\lambda>\frac{c_1}{m}$ and $m\ge m_1$. Thus Lemma~\ref{phimbound} implies that there exists $u\in C^{1}(\mathbb{T})$ such that
\begin{align}\label{starshape112}
\partial D = \left\{ (1+u(\theta))(\cos\theta,\sin\theta) : \theta\in \mathbb{T}\right\} \quad \text{ and }  \quad \rVert u \rVert_{L^{\infty}(\mathbb{T})}<\frac{c_2}{m}.
\end{align}
Since $m\ge m_0 > 2c_2 $, which follows from \eqref{numbers2}, we have that $\rVert u\rVert_{L^{\infty}(\mathbb{T})} < \frac{1}{2}$.

 To derive a contradiction, we will use the identity~\eqref{variationalidentity1}. To estimate the right-hand side of it, we apply Proposition~\ref{l1andl2bound} with $a:=\frac{2c_2}{m} \in \left( 2\rVert u\rVert_{L^{\infty}(\mathbb{T})},1\right)$ and obtain
\begin{align}\label{equation990}
\int_D |x-2\nabla\left( 1_D * \mathcal{N}\right)|^2 dx \le \frac{2c_2\delta}{m}\int_{\mathbb{T}}|u|^2d\theta + \frac{\delta m}{2c_2}\int_{\mathbb{T}}f(\theta)^2d\theta \nonumber\\
\le \frac{2c_2\delta}{m}\int_{\mathbb{T}}|u|^2d\theta + \frac{\pi\delta m}{c_2}\rVert f \rVert_{L^{\infty}(\mathbb{T})}^2,
\end{align}
where $f(\theta) = \int_{0}^{\theta} u(s)^2+2u(s)ds$. Using \eqref{l1equall2} and $\frac{2\pi}{m}$-periodicity of $u$, it is clear that $f$ is also $\frac{2\pi}{m}$-periodic. Furthermore, for $\theta\in (0,\frac{2\pi}{m})$, we have that (recall that $\rVert u \rVert_{L^{\infty}(\mathbb{T})} < \frac{1}{2}$),
\begin{align*}
| f(\theta) |& = \bigg|\int_0^{\theta}u(s)^2+2u(s)ds \bigg| \le \int_0^{\theta}3 |u(s)|ds\le 3\sqrt{\int_{0}^{\theta}|u(s)|^2ds}\sqrt{\theta} < 3\sqrt{\int_0^{\frac{2\pi}{m}}|u(s)|^2ds}\sqrt{\frac{2\pi}{m}}\\
 &\le 3\frac{\sqrt{2\pi}}{m}\sqrt{\int_{\mathbb{T}}|u(s)|^2ds}.
 \end{align*}
 Thus, \eqref{equation990} yields that
 \begin{align}\label{equation992}
 \int_D |x-2\nabla\left( 1_D * \mathcal{N}\right)|^2 dx \le \frac{2c_2\delta}{m}\int_{\mathbb{T}}|u|^2d\theta +\frac{18\pi^2\delta }{c_2m} \int_{\mathbb{T}}|u|^2d\theta = \frac{1}{m}\left( 2c_2\delta+\frac{18\pi^2\delta}{c_2}\right)\int_{\mathbb{T}}|u|^2d\theta.
 \end{align}
 For the left-hand side of \eqref{variationalidentity1}, we use \eqref{secondmoment2} to obtain
 \begin{align}\label{equation991}
 \frac{2}{c_3}\int_{\mathbb{T}}|u|^2d\theta \le \int_{D} {|x|^2}dx-\frac{|D|^2}{2\pi}.
 \end{align}
Hence it follows from \eqref{equation992}, \eqref{equation991} and \eqref{variationalidentity1} that
\begin{align*}
\frac{2\lambda}{c_3}\int_{\mathbb{T}}|u|^2d\theta & \le \lambda \left(\int_{D} {|x|^2}dx-\frac{|D|^2}{2\pi}\right) = \frac{1}{2}\int_D |x-2\nabla\left( 1_D * \mathcal{N}\right)|^2 dx \le \frac{1}{m}\left( c_2\delta+\frac{9\pi^2\delta}{c_2}\right)\int_{\mathbb{T}}|u|^2d\theta.
\end{align*}
Therefore we have
\begin{align*}
\lambda \le \frac{c_3}{2}\left(c_2\delta+\frac{9\pi^2\delta}{c_2}\right)\frac{1}{m}\le \frac{c}{m},
\end{align*}
where the last inequality follows from our choice for $c$ in  \eqref{numbers2}. This contradicts our assumption \eqref{contra1}, thus completes the proof.
\end{proofthm}

By simple maximum principle type argument, Theorem~\ref{largem2} gives a upper bound for $r_{max}$.
  \begin{corollary}\label{rmaxbound}
  There exist constants $c>0$ and $m_1 \ge 3$ such that if $(D,\Omega)$ is a solution to \eqref{rotatingpatch} that is simply-connected, $m$-fold symmetric for some $m\ge m_1$ and $|D| = \pi$, then $r_{max}-1 \le \frac{c}{m}$.

 \end{corollary}   
 \begin{proof}
 Thanks to Theorem~\ref{largem2}, we can pick a constants $C_1$ and ${m}_0$ such that if $m\ge {m}_0$, then $\lambda < \frac{C_1}{m}$. Moreover, it follows from Lemma~\ref{derivatives} that there exists $C_2>0$ such that $|\nabla \varphi_m(r,\theta)| \le \frac{C_2r}{m}$. Now, let us choose
 \begin{align*}
 m_1 := \max\left\{{m}_0, 2\left(C_1+C_2\right)\right\}+1.
 \end{align*}
 Since $\Delta \Psi = 2\lambda >0$ in $D$, the maximum principle for subharmonic functions implies that $\partial_r \Psi(r_{max},0)>0$. Thus it follows from Lemma~\ref{derivatives}  that 
 \begin{align*}
 0 &<  \partial_r \varphi^{r}(r_{max}) +C_2\frac{r_{max}}{m}-\Omega r_{max} \\
 &= \frac{|D\cap B_{r_{max}}|}{2\pi r_{max}}+C_2\frac{r_{max}}{m}-\Omega r_{max}\\
 &= \frac{1}{2r_{max}}+\left(\frac{C_2}{m}-\frac{1}{2}+\lambda \right) r_{max} \\
 & \le \frac{1}{2r_{max}}+\left(\frac{C_1+C_2}{m}-\frac{1}{2}\right)r_{max},
 \end{align*}
 where we used $|D\cap B_{r_{max}}| = |D| = \pi$ to get the second equality and the last inequality follows from $\lambda \le \frac{C_1}{m}$.
 Since $\frac{C_1+C_2}{m}<\frac{1}{2}$, we obtain,
 \begin{align*}
 r_{max} -1 \le \sqrt{\frac{\frac{1}{2}}{\frac{1}{2}-\frac{C_1+C_2}{m}}} -1 \lesssim \frac{1}{m}.
 \end{align*}
 \end{proof}


\subsection{Patches along bifurcation curves}\label{proofoftheorem3}
This subsection is devoted to the proof of Theorem~\ref{largem}. Since we are interested in a curve $\mathscr{C}_m$ that satisfies (A1)-(A4), we will make the following assumptions for the patches throughout this subsection.

\begin{itemize}
\item[(a)] $D$ is star-shaped, that is $\partial D = \left\{ (1+u(\theta))(\cos\theta,\sin\theta) : \theta\in \mathbb{T}\right\}$ for some $u\in C^{2}(\mathbb{T})$. 
\item[(b)] $u$ is even and  $\frac{2\pi}{m}$-periodic for some $m\ge 3$, that is,  $u(-\theta)=u(\theta)$ and $u(\theta+\frac{2\pi}{m})=u(\theta)$.
\item[(c)] $\partial_\theta u(\theta) < 0$ for all $\theta \in (0,\frac{\pi}{m})$. 
\end{itemize}

For such a patch, we denote $r_{min}:=\min_{x\in \partial D}|x| = u(\frac{\pi}{m})$ and $r_{max}:=\max_{x\in \partial D}|x| = u(0)$. Furthermore, we denote  $\eta :=(1+u)^{-1} : (r_{min},r_{max}) \mapsto (0,\frac{\pi}{m})$. By the symmetry, we only need to focus on the fundamental sector $S:=\left\{ (r,\theta) : r\ge 0, \ \theta\in (0,\frac{\pi}{m})\right\}$.  See Figure~\ref{examplepatch} for an illustration of these definitions.

 \begin{figure}[htbp]
\begin{center}
    \includegraphics[scale=0.55]{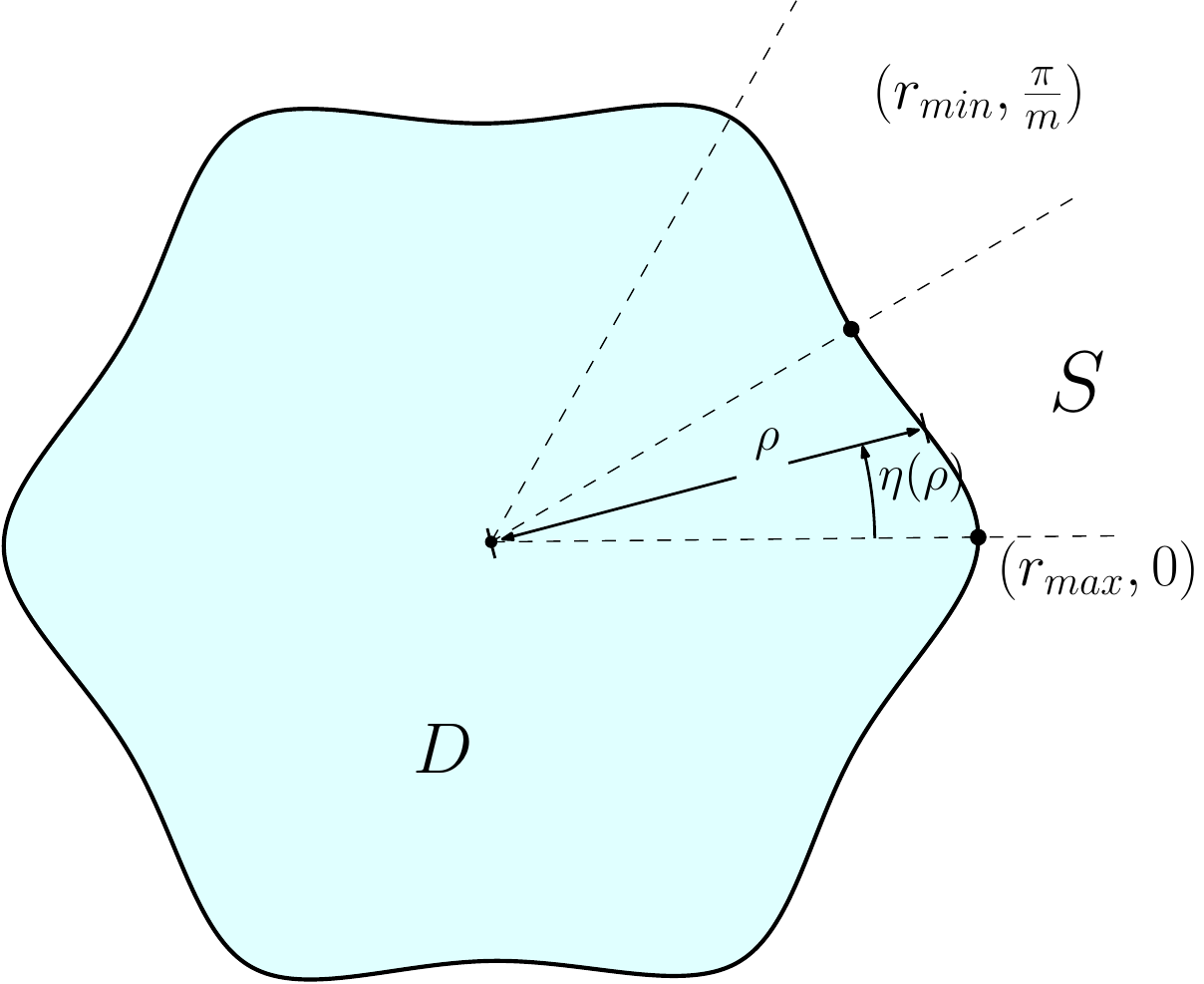}
    \caption{Illustration of the definitions of $r_{min}$, $r_{max}$, $\eta(\rho)$ and $S$ on a $6$-fold vortex patch}
   \label{examplepatch}
    \end{center}

\end{figure}

  Note that we will establish several lemmas with assuming $|D| = |B| = \pi$. Certainly this is not satisfied by the solutions on the curve $\mathscr{C}_m$ but we will resolve this issue in the proof of the theorem . 
  
   Our proof for Theorem~\ref{largem} relies on Theorem~\ref{largem2}. Roughly speaking, we will show that if $\rVert u \rVert_{L^{\infty}(\mathbb{T})}$ is large compared to $\frac{1}{m}$, then $\lambda$ $( =\frac{1}{2}-\Omega)$ must be large enough to contradict Theorem~\ref{largem2}. However, the main difficulty comes from the fact that lower bounds for $\lambda$ that we can derive from the identities \eqref{firstvariation3} and \eqref{firstvariation2} are not comparable with  $\rVert u \rVert_{L^{\infty}}$ (Lemma~\ref{lambdaestimate}). Thus, the scenario that we want to rule out is that for large $m$, $\partial D$ is so spiky that $\int_{\mathbb{T}}|u|^2d\theta$ is small while $\rVert u\rVert_{L^{\infty}}$ is  relatively large.

Since $r_{\max}$ can be estimated as in Corollary~\ref{rmaxbound}, we will mainly focus on estimating $r_{min}$. Using the identity~\eqref{variationalidentity2}, we derive a lower bound for $\lambda$ in the next lemma.

 \begin{lemma}\label{lambdaestimate}
  If $(D,\Omega)$ is a solution to \eqref{rotatingpatch} with $|D| = \pi$ then 
 \begin{align*}
 \lambda \gtrsim \frac{\int_{\mathbb{T}}|u|^2d\theta}{r_{max}\rVert u\rVert_{L^{\infty}(\mathbb{T})}}.
 \end{align*}
 \end{lemma}
 
 \begin{proof}
 we use \eqref{variationalidentity2} in Lemma~\ref{firstvariations} to obtain
\begin{align*}
\lambda\left(\int_D |x|^2-2p(x)dx\right) = \left( \int_D \frac{|x|^2}{2}dx - \frac{|D|^2}{4\pi}\right).
\end{align*}

For a moment, let us assume that
\begin{align}\label{xandpdifference}
\int_D |x|^2 -2p(x)dx \lesssim r_{max}\rVert u\rVert_{L^{\infty}(\mathbb{T})}.
\end{align}
Then it follows from \eqref{secondmoment2} that
\begin{align*}
\lambda = \frac {\int_D \frac{|x|^2}{2}dx - \frac{|D|^2}{4\pi}}{\int_D |x|^2-2p(x)dx} \gtrsim \frac{\int_{\mathbb{T}}|u|^2d\theta}{r_{max}\rVert u\rVert_{L^\infty}(\mathbb{T})},
\end{align*}
which implies the desired result. 

Now let us prove \eqref{xandpdifference}. Note that $p(x) \ge \frac{r_{min}^2 - |x|^2}{2}$ in $B_{r_{min}}$. Indeed, $p-\frac{r_{min}^2 - |x|^2}{2}$ is harmonic in $B_{r_{min}}$ and non-negative on $\partial B_{r_{min}}$ since $p$ is non-negative in $D$. Therefore the inequality follows from the maximum principle.  
From this,  we obtain
\begin{align*}
\int_{D}2p(x)dx \ge \int_{B_{r_{min}}}r_{min}^2 - |x|^2 dx = \frac{|B_{r_{min}}|^2}{2\pi}.
\end{align*}
Since $\int_{D}|x|^2dx < \int_{B_{r_{max}}}|x|^2dx = \frac{|B_{r_{max}}|^2}{2\pi}$, it follows that
\begin{align*}
\int_{D}|x|^2 - 2p(x)dx = \frac{1}{2\pi}\left( |B_{r_{max}}|^2 - |B_{r_{min}}|^2 \right) \lesssim r_{max}\rVert u \rVert_{L^{\infty}(\mathbb{T})},
\end{align*}
which proves \eqref{xandpdifference}.
 \end{proof}
 
  Thanks to Lemma~\ref{lambdaestimate}, we only need to rule out the case where $\int_{\mathbb{T}} |u|^2 d\theta$ is too small, compared to $\rVert u \rVert_{L^{\infty}}$. To this end, we will pick $r_{1}$, and  $r_{2}$ so that $r_{min} < r_1 < r_2 < 1$ and find a lower bound for  $\frac{\pi}{m} - \eta(r_2)$ by showing that $|u'(\theta)|$ is bounded from above for $1+u(\theta) \in (r_{min},r_1)$.   Since the relative stream function $\Psi$ is constant on $\partial D$, we have $\frac{d}{d\theta}\left(\Psi((1+u(\theta)),\theta)\right) =0$. Therefore \eqref{decomp} yields that
  \begin{align}\label{derivativeofu112}
  u'(\theta) = -\frac{\partial_\theta\Psi(r,\theta)}{\partial_r\Psi(r,\theta)} = - \frac{\partial_{\theta}\varphi_m(r,\theta)}{\left(\partial_r \varphi^r(r)-\Omega r\right) + \partial_r\varphi_m(r,\theta)},\ \text{ where }r=1+u(\theta).
  \end{align}
 In the next two lemmas, we will estimate the denominator and numerator in \eqref{derivativeofu112} but the proofs will be postponed to the end of this subsection.

 \begin{lemma}\label{boundforvarphi2} 
Let $(D,\Omega)$ be a solution to \eqref{rotatingpatch} that satisfies the assumptions (a)-(c) for some $m\ge 3$ and $|D| = \pi$. Let $r_1$, $r_2 >0$ be such that $r_{min} < r_1 < r_2 <1$ and let $\delta := \frac{\pi}{m}-\eta(r_2)$. Then there exist constants $c,$ $C>0$ such that if $\rVert u\rVert_{L^{\infty}(\mathbb{T})} \le \frac{1}{2}$, it holds that 
 \begin{align*}
 \partial_r \varphi^r(r) - \Omega r \ge m\delta\left( c\frac{(1-r_2)^2}{\rVert u\rVert_{L^{\infty}(\mathbb{T})}} - C\left(r_1-r_{min}\right)\right),
 \end{align*}
 for all $r\in (r_{min},r_1)$.
 \end{lemma}

\begin{lemma}\label{derivativeestimates}
Let $D$ be a patch that satisfies the assumptions (a)-(c) for some $m\ge 3$. Let  us pick $r_1$, $r_2>0$ so that $r_{min} < r_1 < r_2<1$ and $r_1^2 \le r_{min}r_2$. If $\delta:=\frac{\pi}{m}-\eta(r_2) < \frac{\pi}{4m}$, then it holds that
\begin{align}
\partial_r \varphi_m(r,\eta(r)) \ge -\frac{cr}{1-\left(\frac{r_1}{r_2}\right)^m}\delta \label{radial1} \\
\partial_\theta \varphi_m(r,\eta(r))  \le \frac{cr^2}{1-\left(\frac{r_1}{r_2}\right)^m}\delta. \label{angular1}
\end{align}
for all $r \in \left( r_{min}, r_1\right)$, where $c$ is a universal constant that does not depend on any variables.. 
\end{lemma} 
Note that the linear dependence on $\delta$ in \eqref{radial1} and \eqref{angular1} is crucial in the proof of the next lemma, since this allows us to bound $u'$ independently of $\delta$ when we plug the above bounds into \eqref{derivativeofu112}.

 Now we can rule out the scenario that $\partial D$ is too spiky inwards.
 
 \begin{lemma}\label{nospike} 
There exist $c,C>0$ and $m_1\ge3$ such that if $(D,\Omega)$ is a solution to \eqref{rotatingpatch} that satisfies the assumptions (a)-(c) for some $m\ge m_1$, $|D| = \pi$ and $\rVert u \rVert_{L^{\infty}(\mathbb{T})}\le \frac{1}{2}$, then 
\begin{align*}
 1-r_{min} \le \frac{c}{m}\quad \text{ or }\quad \int_{\mathbb{T}} |u|^2d\theta \ge C\left( 1-r_{min}\right)^2.
\end{align*}
 \end{lemma}
 \begin{proof}
  Thanks to Lemma~\ref{boundforvarphi2} and ~\ref{derivativeestimates}, we can choose $c_1$ and $c_2>0$ such that if $r_{min}<r_1<r_2<1$ and $r_1^2\le r_{min}r_2$, then
 \begin{subequations}\label{derivatives223}
\begin{alignat}{2}
 \partial_r\varphi_r(r) - \Omega r \ge m\delta\left( c_1\frac{(1-r_2)^2}{\rVert u \rVert_{L^{\infty}(\mathbb{T})}}-c_2\left(r_1-r_{min}\right)\right),\\
 \partial_r \varphi_m(r,\eta(r)) \ge - \frac{c_2}{1-\left(\frac{r_1}{r_2}\right)^m}\delta,\\
  \partial_\theta \varphi_m(r,\eta(r)) \le \frac{c_2}{1-\left(\frac{r_1}{r_2}\right)^m}\delta,
 \end{alignat}
 \end{subequations}
for $r\in (r_{min},r_1)$. We will pick
\begin{align}\label{numbers11232}
c:= \max\left\{12, \frac{24c_2}{c_1}\right\},  \quad  C:=   \frac{\pi}{18} \quad \text{ and }\quad m_1:= 2c
\end{align}
Now let us assume that $D$ is a solution to \eqref{rotatingpatch} that satisfies the assumptions (a)-(c)  for some $m\ge m_1$, $|D| = \pi$ and $1-r_{min} =\rVert u \rVert_{L^{\infty}(\mathbb{T})} < \frac{1}{2}$. If $1-r_{min} \le \frac{c}{m}$, then there is nothing to prove, thus let us assume that 
\begin{align}\label{asmpt1231}
1-r_{min}> \frac{c}{m}. 
\end{align}
 Let $\tilde{c}:=m(1-r_{min}) $ so that $1 - r_{min} = \frac{\tilde{c}}{m}$. From \eqref{asmpt1231}, we have 
 \begin{align}\label{asmpt1232}
 \tilde{c} > c.
 \end{align}
  We choose 
 \begin{align}\label{choiceforr112}
 r_{1}:= 1-\frac{\tilde{c}-2}{m} \quad \text{ and }\quad r_2:=1-\frac{\tilde{c}}{3m}.
 \end{align}
 And we consider two cases: $\frac{\pi}{m} - \eta(r_2) \ge \frac{\pi}{4m}$ and  $\frac{\pi}{m} - \eta(r_2) < \frac{\pi}{4m}$.
 
 \textbf{Case1.} Let us assume that $\frac{\pi}{m} - \eta(r_2) \ge \frac{\pi}{4m}$. 

Since $\frac{\pi}{m}-\eta(r_2) \ge \frac{\pi}{4m}$, it follows from the monotonicity of $u$ that $u(\theta) < r_{2}-1 = -\frac{\tilde{c}}{3m}$ for $\theta \in \left( \frac{3\pi}{4m}, \frac{\pi}{m}\right)$. Using $m$-fold symmetry of $u$, we obtain
\begin{align}\label{case1result}
\int_{\mathbb{T}} |u|^2 d\theta \ge 2m\int_{0}^{\frac{\pi}{m}} |u|^2 d\theta > 2m\int_{\frac{3\pi}{4m}}^{\frac{\pi}{m}} |u|^2 d\theta \ge \frac{\pi\tilde{c}^2}{18m^2} = C(1-r_{min})^2,
\end{align}
where the last equality follows from \eqref{numbers11232}, which says $C= \frac{\pi}{18}$.
 
 \textbf{Case2.} Now we assume $\frac{\pi}{m}-\eta(r_2) < \frac{\pi}{4m}$.
 
 We first check whether $r_1$ and $r_2$ in \eqref{choiceforr112} satisfy the hypotheses in Lemma~\ref{derivativeestimates}. Clearly $r_{min} < r_{1} < {r_2} < 1$, since $\tilde{c} > c \ge 4$,  which follows from \eqref{numbers11232} and \eqref{asmpt1232}.   To show $r_{1}^2 \le r_{min}r_2$, we compute
\begin{align*}
r_{1}^2 - r_{min}r_2 &= \left( 1-\frac{\tilde{c}-2}{m}\right)^2 - \left(1-\frac{\tilde{c}}{m}\right)\left( 1-\frac{\tilde{c}}{3m}\right) \\
&  = -\frac{2\tilde{c}}{3m}+\frac{4}{m}+\frac{2\tilde{c}^2-12\tilde{c}+12}{3m^2}\\
& \le -\frac{\tilde{c}}{3m}+\frac{4}{m}+\frac{-12\tilde{c}+12}{3m^2} \\
& \le 0,
\end{align*}
where the first ineqiality follows from $\frac{\tilde{c}}{m} = 1-r_{min} = \rVert u \rVert_{L^{\infty}(\mathbb{T})} \le \frac{1}{2}$, and the last inequality follows from $\tilde{c} > 12$.

Since $\partial D$ is a level set of $\Psi$, $\Psi(1+u(\theta),\theta) = $ does not depend on $\theta$. Therefore, 
\begin{align*}
-u'(\theta) = \frac{\partial_{\theta}\Psi(1+u(\theta),\theta)}{\partial_r\Psi(1+u(\theta),\theta)} = \frac{\partial_\theta \varphi_m(1+u(\theta),\theta)}{\partial_r \varphi^r(1+u(\theta))-\Omega r+\partial_r\varphi_m(1+u(\theta),\theta)}.
\end{align*}
Hence, it follows from \eqref{derivatives223} that
\begin{align}\label{derivativeofu}
-u'(\theta) \le \frac{\frac{1}{1-\left(\frac{r_1}{r_2}\right)^m}}{\frac{c_1}{c_2}m\frac{(1-r_2)^2}{\rVert u \rVert_{L^{\infty}(\mathbb{T})}} - m\left( r_1 - r_{min}\right) - \frac{1}{1-\left(\frac{r_1}{r_2}\right)^m}} \quad \text{ for } \quad (1+u(\theta)) < r_1.
\end{align}

Now we assume for a moment that
\begin{align}
\frac{\frac{c_1}{c_2}m(1-r_2)^2}{3\rVert u\rVert_{L^{\infty}}} \ge \max\left\{ m(r_1 - r_{min}), \frac{1}{1-\left(\frac{r_1}{r_2}\right)^m}\right\}.\label{claimforuprime}
\end{align}
Then \eqref{derivativeofu} yields that
\begin{align*}
0\le -u'(\theta) \le1 \quad \text{ for } \quad (1+u(\theta))<r_1,
\end{align*}
which implies that
\begin{align}\label{eqn1}
\frac{2}{m} = r_1 - r_{min} = \int_{\eta(r_{1})}^{\frac{\pi}{m}} -u'(\theta)d\theta \le \frac{\pi}{m}-\eta(r_1).
\end{align}
Furthermore, monotonicity of $u$ implies that $|u(\theta)| > |r_1 - 1 | > |r_2-1|$ for $\theta \in \left( \eta(r_1),\frac{\pi}{m}\right)$. Therefore we obtain that
\begin{align}\label{case2result}
\int_{\mathbb{T}} |u|^2d\theta > 2m\int_{\eta(r_1)}^{\frac{\pi}{m}}|u|^2d\theta \ge 2m\left(\frac{\pi}{m}-\eta(r_1)\right) (1-r_2)^2 \ge \frac{4\tilde{c}^2}{9m^2}= \frac{4}{9}\left(1-r_{min}\right)^2 \ge C(1-r_{min})^2,
\end{align}
where the third inequality follows from \eqref{choiceforr112} and \eqref{eqn1} and the last inequality follows from $C = \frac{\pi}{18} < \frac{4}{9}$. Thus the desired result follows from \eqref{case1result} and \eqref{case2result}.

 To complete the proof, we need to show \eqref{claimforuprime}. It follows from  \eqref{choiceforr112} that 
 \begin{align}\label{eqn3}
 \frac{\frac{c_1}{c_2} m \left(1-r_2\right)^2}{3\rVert u \rVert_{L^{\infty}(\mathbb{T})}} =\frac{\frac{c_1}{c_2} m \left(1-r_2\right)^2}{3(1-r_{min})}  = \frac{c_1\tilde{c}}{12c_2} \ge 2,
 \end{align}
 where the last inequality follows from \eqref{numbers11232} and \eqref{asmpt1232} which imply that $\tilde{c} \ge \frac{24c_2}{c_1}$. We also have
 \begin{align}\label{eqn332}
 m\left(r_1 - r_{min}\right) =  2,
 \end{align}
 which follows from \eqref{choiceforr112}. 
 To estimate $\frac{1}{1-\left(\frac{r_1}{r_2}\right)^m}$, let us use an elementary inequality that
 for any $0<b<a<m$, it holds that
 \begin{align}\label{smallclaim}
 \left(\frac{1-\frac{a}{m}}{1-\frac{b}{m}}\right)^{m}\le e^{-(a-b)}.
 \end{align}
 Indeed, by taking logarithm in the left-hand  side, we can compute,
 \begin{align*}
 \log\left( \left(\frac{1-\frac{a}{m}}{1-\frac{b}{m}}\right)^{m}\right) = m\log\left(1-\frac{a-b}{m-b}\right) \le m\log\left(1-\frac{a-b}{m} \right) \le - (a-b),
 \end{align*}
 where the last inequality follows from $\log(1-x)<-x$ for all $x>0$. This proves \eqref{smallclaim}. Then we use \eqref{choiceforr112} and obtain
 \begin{align}\label{eqn337}
 \frac{1}{1-\left(\frac{r_1}{r_2}\right)^m} = \frac{1}{1-\left(\frac{1-\frac{\tilde{c}-2}{m}}{1-\frac{\tilde{c}}{3m}} \right)^m} \le \frac{1}{1-e^{-(\frac{2\tilde{c}}{3}-2)}} \le 2,
 \end{align}
 where the last inequality follows from \eqref{numbers11232} and \eqref{asmpt1232}, which imply $\tilde{c}\ge 12$.

 Thus \eqref{eqn3}, \eqref{eqn332} and \eqref{eqn337}  yield 
 \begin{align}\label{eqn2}
  \frac{\frac{c_1}{c_2} m \left(1-r_2\right)^2}{3\rVert u \rVert_{L^{\infty}(\mathbb{T})}}  \ge   2 \ge \max\left\{ m\left(r_1 - r_{min}\right) ,\frac{1}{1-\left(\frac{r_1}{r_2}\right)^{m}}\right\}.
 \end{align}
 This proves \eqref{claimforuprime} and finishes the proof.
 \end{proof}

Now we can estimate $\rVert u\rVert_{L^{\infty}(\mathbb{T})}$ whose corresponding patch has area $\pi$, that is, $|D| = \pi$. 
\begin{prop}\label{areaispi}
There exist constants $c>0$ and $m_0 \ge 3$ such that if $(D,\Omega)$ is a solution to \eqref{rotatingpatch} that satisfies assumptions (a)-(c) for some $m\ge m_0$, $|D| = \pi$ and $\rVert u \rVert_{L^{\infty}(\mathbb{T})} \le \frac{1}{2}$, then
\begin{align*}
\rVert u \rVert_{L^{\infty}(\mathbb{T})} \le \frac{c}{m}.
\end{align*}
\end{prop}

\begin{proof}
In order to use the previous lemmas, let us fix some constants. We fix constants $c_i's$ and $m_1$ so that if $(D,\Omega)$ is a solution to \eqref{rotatingpatch} that satisfies the assumptions (a)-(c) for some $m\ge m_1$ and $|D| = \pi$, then 
\begin{itemize}
\item[(B1)] (From Theorem~\ref{largem2}) $\lambda \le \frac{c_1}{m}$.
\item [(B2)] (From Corollary~\ref{rmaxbound}) $r_{max}-1 \le \frac{c_2}{m}$.
\item[(B3)] (From Lemma~\ref{lambdaestimate}) $\lambda \ge c_3\frac{\int_{\mathbb{T}}|u|^2d\theta}{r_{max}\rVert u \rVert_{L^{\infty}(\mathbb{T})}}$.
\item[(B4)] (From Lemma~\ref{nospike}) if $1-r_{min}=\rVert u \rVert_{L^{\infty}(\mathbb{T})}\le \frac{1}{2}$ then
\begin{align*}
1-r_{min} \le \frac{c_4}{m} \quad \text{ or }\quad \int_{\mathbb{T}}|u|^2d\theta\ge c_5(1-r_{min})^2.
\end{align*}
\end{itemize}

Let us set
\begin{align}\label{numbers1123}
c:=\max\left\{ c_2, c_4, \frac{2c_1}{c_3c_5}\right\}+1, \quad \text{ and }\quad m_0 :=\max\left\{ m_1, 2c \right\}+1.
\end{align}
Then we will prove that if $(D,\Omega)$ is a solution to \eqref{rotatingpatch} and satisfies the assumptions (a)-(c) for some $m\ge m_0$, then
\begin{align}\label{goal112}
\rVert u\rVert_{L^{\infty}(\mathbb{T})} \le \frac{c}{m}.
\end{align}

 Let us assume for a contradiction that 
 \begin{align}\label{contrapositive1123}
 \rVert u \rVert_{L^{\infty}(\T)} > \frac{c}{m}.
 \end{align}
  Then we have that
  \begin{align}\label{eqn3381}
 \int_{\mathbb{T}} |u|^2d\theta \ge c_5 (1-r_{min})^2. 
 \end{align}
 Indeed, 
  \begin{align*}
 \rVert u \rVert_{L^{\infty}(\mathbb{T})} > \frac{c}{m} > \frac{c_2}{m} \ge r_{max}-1,
 \end{align*}
  where we used \eqref{numbers1123}, \eqref{contrapositive1123} and (B2), therefore $\rVert u \rVert_{L^{\infty}(\mathbb{T})} = 1-r_{min}$. Thus  (B4) and \eqref{contrapositive1123} imply \eqref{eqn3381}.
 Furthermore (B2) and \eqref{numbers1123} also imply that 
 \begin{align*} 
 r_{max} \le 1+\frac{c_2}{m}\le 2.
 \end{align*}
 Thus we use (B3) and \eqref{eqn3381} to obtain
 \begin{align*}
 \lambda \ge c_3 \frac{c_5 (1-r_{min})^2}{2\rVert u \rVert_{L^{\infty}(\mathbb{T})}} \ge \frac{c_3c_5\rVert u \rVert_{L^{\infty}(\mathbb{T})}}{2} \ge \frac{cc_3c_5}{2m} > \frac{c_1}{m},
 \end{align*}
 where the third inequality follows from \eqref{contrapositive1123} and the last inequality follows from \eqref{numbers1123}. However this contradicts (B1). 
 \end{proof}

Now we are ready to prove the main theorem of this subsection
\begin{proofthm}{largem}
Thanks to Proposition~\ref{areaispi}, we can pick $c_1$ and $m_1$ so that $2c_1 < m_1$ and if $(D,\Omega)$ is a solution to \eqref{rotatingpatch}, that satisfies the assumptions (a)-(c) for some $m\ge m_1$ and $\rVert u \rVert_{L^{\infty}(\mathbb{T})} \le \frac{1}{2}$, then 
\begin{align}\label{prop11}
\rVert u \rVert_{L^{\infty}(\T)} \le \frac{c_1}{m}.
\end{align}
Now let us consider a curve $\mathscr{C}_m$, that satisfies the properties (A1)-(A4) for some $m\ge m_1$. We will show that
\begin{align}\label{goal11}
\sup_{s\in [0,\infty)}\rVert \tilde{u}_m(s)\rVert_{L^{\infty}(\mathbb{T})} \lesssim \frac{1}{m}.
\end{align} 
To do so, let us define $\hat{u}_m(s)$ so that, 
\begin{align}\label{defofuhat}
1+\hat{u}_m(s) = \frac{\sqrt{\pi}}{\sqrt{|D^{\tilde{u}_m(s)}|}} (1+\tilde{u}_m(s)),
\end{align}
 where the definition of $D^{\tilde{u}_m(s)}$ is as in (A2). Clearly, $s \mapsto \hat{u}_m(s)$ is a continuous curve in $C^{2}(\mathbb{T})$ such that $\left| D^{\hat{u}_m(s)} \right| = \pi$. Since $\hat{u}_m(0) = 0$, it follows from the continuity of the curve and \eqref{prop11} that
\begin{align}\label{prop12}
\sup_{s\in [0,\infty)} \rVert \hat{u}_m(s)\rVert_{L^{\infty}(\mathbb{T})} \le \frac{c_1}{m}.
\end{align}
Now let us pick an arbitrary $s\in [0,\infty)$ and denote $\tilde{u}:=\tilde{u}_m(s)$ and $\hat{u}:=\hat{u}_m(s)$. Then it follows from (A1) and \eqref{defofuhat} that
\begin{align*}
0 = \int_{\mathbb{T}} \tilde{u}(\theta)d\theta  & =  2\pi \left( \frac{\sqrt{\left| D^{\tilde{u}}\right|}}{\sqrt{\pi}}-1\right) + \frac{\sqrt{\left| D^{\tilde{u}}\right|}}{\sqrt{\pi}}\int_{\mathbb{T}} \hat{u}(\theta)d\theta \\
& =  2\pi \left( \frac{\sqrt{\left| D^{\tilde{u}}\right|}}{\sqrt{\pi}}-1\right) - \frac{1}{2}\frac{\sqrt{\left| D^{\tilde{u}}\right|}}{\sqrt{\pi}}\int_{\mathbb{T}} \hat{u}(\theta)^2 d\theta,
\end{align*}
where the last equality follows from \eqref{l1equall2}. Hence \eqref{prop12} implies that
$\frac{\sqrt{|D^{\tilde{u}}|}}{\sqrt{\pi}} = 1 + O\left(\frac{1}{m^2} \right)$. Therefore \eqref{defofuhat} and \eqref{prop12} yield that
\begin{align*}
\rVert \tilde{u} \rVert_{L^{\infty}(\mathbb{T})} = \frac{\sqrt{|D^{\tilde{u}}|}}{\sqrt{\pi}} (1+\rVert \hat{u} \rVert_{L^{\infty}(\mathbb{T})}) - 1 = \rVert \hat{u} \rVert_{L^{\infty}(\mathbb{T})} + O\left(\frac{1}{m^2}\right) \lesssim \frac{1}{m},
\end{align*}
where the last equality follows from \eqref{prop12}. This proves \eqref{goal11} and the theorem.
 \end{proofthm}

\subsubsection{Proofs of Lemma~\ref{boundforvarphi2} and ~\ref{derivativeestimates}}\label{subsection321}
 \begin{prooflem}{boundforvarphi2}
 From Lemma~\ref{derivatives}, it follows that
\begin{align}\label{radialpart1123}
 \partial_r \varphi^r(r) - \Omega r = r\left( \frac{1}{2}-\Omega\right) - \frac{|B_r\backslash D|}{2\pi r} = r\lambda -\frac{|B_r \backslash D|}{2\pi r} =: J_1 - J_2,
\end{align}
where we used $\lambda = \frac{1}{2}-\Omega$. Note that we have $\rVert u \rVert_{L^{\infty}(\mathbb{T})}\le \frac{1}{2}$, therefore $J_1 \sim \lambda$ and $J_2\sim |B_r \backslash D|$.  Let us estimate $J_2$ first. Since $u$ is even and $m$-periodic, we have that for all $r\in (r_{min},r_1)$,
\begin{align*}
|B_r\backslash D| = 2m\int_{r_{min}}^{r} \left(\frac{\pi}{m}-\eta(\rho)\right)\rho d\rho \lesssim mr\delta(r-r_{min}) \le mr\delta(r_1-r_{min}),
\end{align*}
where we used $\frac{\pi}{m}-\eta(\rho) < \delta$ for $\rho <r<r_{2}$ to get the first inequality. Hence we obtain
\begin{align}\label{J2}
J_2 \lesssim m\delta\left(r_1-r_{min}\right).
\end{align} 
To estimate $J_1$, we use Lemma~\ref{lambdaestimate} and obtain
\begin{align}\label{lambdaestimate112}
\lambda \gtrsim \frac{\int_{\mathbb{T}}|u|^2d\theta}{r_{max}\rVert u \rVert_{L^{\infty}}} \gtrsim \frac{\int_{\mathbb{T}}|u|^2d\theta}{\rVert u \rVert_{L^{\infty}}},
\end{align}
where we we used $r_{max} \le 1+\rVert u \rVert_{L^{\infty}(\mathbb{T})} \lesssim 1$ to get the last inequality. From periodicity of $u$, it follows that
\begin{align}\label{lambdaestimate113}
\int_{\mathbb{T}} |u|^2d\theta = 2m\int_{0}^{\frac{\pi}{m}}|u|^2d\theta \ge 2m \int_{\eta(r_2)}^{\frac{\pi}{m}} |u|^2 d\theta \ge 2m (1-r_2)^2\left(\frac{\pi}{m}-\eta(r_2)\right) = 2m\delta(1-r_2)^2,
\end{align}
where we used $1+u(\theta) < r_2$ for $\theta\in (\eta(r_2),\frac{\pi}{m})$ by monotonicity of $u$ to get the second inequality. Hence it follows from \eqref{lambdaestimate112} and \eqref{lambdaestimate113} that 
\begin{align}\label{J1}
J_1 \gtrsim \lambda \gtrsim \frac{m\delta(1-r_2)^2}{\rVert u \rVert_{L^{\infty}(\mathbb{T})}}.
\end{align}
Thus the desired result follows from \eqref{radialpart1123}, \eqref{J2} and \eqref{J1}.
 \end{prooflem}

Now we prove Lemma~\ref{derivativeestimates}. The proof is based on the formulae given in \eqref{derivatives1}.
 \begin{prooflem}{derivativeestimates}
Let us assume that $\delta < \frac{\pi}{4m}$. We will prove \eqref{radial1} first. By monotonicity of $\eta$ (assumption (c)), it follows from Lemma~\ref{derivatives1} that for all $r\in \left( r_{min},r_1\right)$, 
 \begin{align*}
 \partial_r & \varphi_m(r,\eta(r))  \ge \frac{1}{2\pi}\int_{r_{min}}^{r} \frac{\rho}{r} \arctan{\left(\frac{(\frac{\rho}{r})^m\sin{(m(\eta(r)+\eta(\rho)))}}{1-\left(\frac{\rho}{r}\right)^m\cos{(m(\eta(r)+\eta(\rho)))}}\right)} d\rho \\
  & \ + \frac{1}{2\pi} \int_{r}^{r_{max}} \frac{\rho}{r} \left(\arctan{\left(\frac{(\frac{r}{\rho})^m\sin{(m(\eta(r)-\eta(\rho)))}}{1-\left(\frac{r}{\rho}\right)^m\cos{(m(\eta(r)-\eta(\rho)))}}\right)}- \arctan{\left(\frac{(\frac{r}{\rho})^m\sin{(m(\eta(r)+\eta(\rho)))}}{1-\left(\frac{r}{\rho}\right)^m\cos{(m(\eta(r)+\eta(\rho)))}}\right)}\right) d\rho,
 \end{align*} 
 where we used that $\sin(m(\eta(r)-\eta(\rho))) \le 0$ for $\rho < r$ so we can drop one of the integrands for free. Note that the integrand in the second integral is positive for $\rho\in (r,r_2)$, since $\sin(m(\eta(r)-\eta(\rho))) >0$ and $\sin(m(\eta(r)+\eta(\rho))) <0$ for $r<\rho<r_2$, which follows from $\frac{\pi}{m}-\eta(r_2) < \frac{\pi}{4m}$. We will use the second integrand in the second integral to cancel the first integral, that is, we have that
 \begin{align}\label{last112}
 & \partial_r  \varphi_m(r,\eta(r)) \ge \nonumber \\
 &  \frac{1}{2\pi} \left( \int_{r_{min}}^{r} \frac{\rho}{r} \arctan{\left(\frac{(\frac{\rho}{r})^m\sin{(m(\eta(r)+\eta(\rho)))}}{1-\left(\frac{\rho}{r}\right)^m\cos{(m(\eta(r)+\eta(\rho)))}}\right)} d\rho - \int_{r}^{r_2} \frac{\rho}{r}\arctan{\left(\frac{(\frac{r}{\rho})^m\sin{(m(\eta(r)+\eta(\rho)))}}{1-\left(\frac{r}{\rho}\right)^m\cos{(m(\eta(r)+\eta(\rho)))}}\right)} d\rho\right) \nonumber \\
 & + \frac{1}{2\pi} \int_{r_2}^{r_{max}} \frac{\rho}{r} \left(\arctan{\left(\frac{(\frac{r}{\rho})^m\sin{(m(\eta(r)-\eta(\rho)))}}{1-\left(\frac{r}{\rho}\right)^m\cos{(m(\eta(r)-\eta(\rho)))}}\right)}- \arctan{\left(\frac{(\frac{r}{\rho})^m\sin{(m(\eta(r)+\eta(\rho)))}}{1-\left(\frac{r}{\rho}\right)^m\cos{(m(\eta(r)+\eta(\rho)))}}\right)}\right)d\rho \nonumber \\
 & =: \frac{1}{2\pi} K_1 + \frac{1}{2\pi} K_2.
 \end{align}
 To estimate $K_1$, we use that (note that $\frac{\pi}{m}-\eta(r)<\frac{\pi}{4m}$ for $r\le r_2$)
 \begin{align*}
 \begin{cases}
 \sin(m(\eta(r)+\eta(\rho))) \ge \sin(2m\eta(r)) \quad \text{ and } \quad \cos(m(\eta(r)+\eta(\rho))) \le 1 & \text{ for }\rho\in (r_{min},r) \\
 \sin(m(\eta(r)+\eta(\rho))) \le \sin(2m\eta(r)) \quad \text{ and } \quad \cos(m(\eta(r)+\eta(\rho))) \ge \cos(2m\delta) & \text{ for }\rho\in (r,r_2),
 \end{cases}
 \end{align*}
 and obtain
 \begin{align*}
 K_1 \ge \int_{r_{min}}^{r} \frac{\rho}{r} \arctan{\left(\frac{(\frac{\rho}{r})^m\sin{(2m\eta(r))}}{1-\left(\frac{\rho}{r}\right)^m}\right)} d\rho -\int_r^{r_2} \frac{\rho}{r}\arctan{\left(\frac{(\frac{r}{\rho})^m\sin{(2m\eta(r))}}{1-\left(\frac{r}{\rho}\right)^m\cos{(2m\delta)}}\right)} d\rho.
 \end{align*}
 From the change of variables $\left(\frac{\rho}{r}\right)^m \mapsto x$ for the first integral and $\left(\frac{r}{\rho}\right)^m \mapsto x$ for the second integral, it follows  that
 \begin{align*}
 K_1 & \ge \frac{r}{m}\left(\int_{(\frac{r_{min}}{r})^m}^1 x^{-1+\frac{2}{m}}\arctan\left(\frac{x\sin(2m\eta(r))}{1-x}\right)dx -\int_{{\left(\frac{r}{r_2}\right)}^m}^{1} x^{-1-\frac{2}{m}}\arctan\left( \frac{x\sin(2m\eta(r))}{1-x\cos(2m\delta)}\right)dx\right) \\
  &\ge \frac{r}{m}\int_{\left(\frac{r}{r_2}\right)^m}^{1}x^{-1-\frac{2}{m}}\left(\arctan\left( \frac{x\sin(2m\eta(r))}{1-x}\right)-\arctan\left( \frac{x\sin(2m\eta(r))}{1-x\cos(2m\delta)}\right)\right)dx.
 \end{align*}
 where we used $\frac{r_{min}}{r} > \frac{r}{r_2}$ for $r\in (r_{min},r_1)$, which follows from $r_1^2 < r_{min}r_{2}$, and $x^{-1-\frac{2}{m}}>x^{-1+\frac{2}{m}}$  for $0<x<1$ to get the second inequality (note that the first integrand is negative). Therefore it follows from Lemma~\ref{inequality1} that (note that $\sin{(2m\eta(r))}<0$ for $r<r_1$),
 \begin{align}\label{K1bound}
 K_1 \gtrsim -\frac{r}{m} (1-\cos(2m\delta)) \gtrsim -\frac{r}{m}\sin(2m \delta) \gtrsim -r\delta,
 \end{align}
 where the second inequality is due to $1-\cos(2m\delta) < \sin(2m\delta)$ for $\delta <\frac{\pi}{4m}$. \\

 Now let us estimate $K_2$. Note that the integrand in $K_2$ is non-negative if $\eta(\rho) > \frac{\pi}{m}-\eta(r)$. Indeed, by monotonicity of $\eta$, we have $\eta(r) > \eta(\rho)$ for any $\rho \ge r_2 > r$, which implies the first integrand in $K_2$ is positive for all $\rho \in (r_2,r_{max})$. Thus, if we choose $r_3:=\eta^{-1}(\frac{\pi}{m}-\eta(r))>r_2$ then the integrand in $K_2$ is strictly positive for $\rho \in (r_2,r_3)$. Hence we have
 \begin{align*}
 K_2   \ge \int_{r_3}^{r_{max}} \frac{\rho}{r} \left(\arctan{\left(\frac{(\frac{r}{\rho})^m\sin{(m(\eta(r)-\eta(\rho)))}}{1-\left(\frac{r}{\rho}\right)^m\cos{(m(\eta(r)-\eta(\rho))}}\right)}- \arctan{\left(\frac{(\frac{r}{\rho})^m\sin{(m(\eta(r)+\eta(\rho)))}}{1-\left(\frac{r}{\rho}\right)^m\cos{(m(\eta(r)+\eta(\rho)))}}\right)}\right)d\rho.
 \end{align*}
Note that $\bigg|\partial_\theta \left( \arctan\left(\frac{x\sin(m\theta)}{1-x\cos(m\theta)}\right)\right)\bigg| \lesssim \frac{mx}{1-x}$ for all $0\le x<1$ and $\theta\in [-\frac{\pi}{m},\frac{\pi}{m}]$. Indeed,
\begin{align*}
\partial_\theta \left( \arctan\left(\frac{x\sin(m\theta)}{1-x\cos(m\theta)}\right)\right) = \frac{mx(\cos(m\theta)-1+(1-x))}{(x-1)^2+2x(1-\cos(m\theta))}.
\end{align*}
Hence either $(1-x) \le 1-\cos(m\theta)$ or $(1-x) > 1-\cos(m\theta)$, one can easily see that $\bigg| \frac{mx(\cos(m\theta)-1+(1-x))}{(x-1)^2+2x(1-\cos(m\theta))} \bigg| \lesssim \frac{mx}{1-x}$. Since we also have $\frac{r}{\rho} <\frac{r_1}{r_3}<\frac{r_1}{r_2}$ for $\rho > r_3$, and  it follows from the mean-value theorem that the integrand can be bounded as
\begin{align*}
\bigg|\arctan{\left(\frac{(\frac{r}{\rho})^m\sin{(m(\eta(r)-\eta(\rho)))}}{1-\left(\frac{r}{\rho}\right)^m\cos{(m(\eta(r)-\eta(\rho))}}\right)}- \arctan{\left(\frac{(\frac{r}{\rho})^m\sin{(m(\eta(r)+\eta(\rho)))}}{1-\left(\frac{r}{\rho}\right)^m\cos{(m(\eta(r)+\eta(\rho)))}}\right)}\bigg| & \lesssim \frac{m\eta(\rho)\left(\frac{r}{\rho}\right)^{m}}{1-\left(\frac{r}{\rho}\right)^{m}} \\
& < \frac{m\eta(\rho)\left(\frac{r}{\rho}\right)^m}{1-\left(\frac{r_1}{r_2}\right)^{m}}.
\end{align*}
Therefore we obtain
\begin{align*}
K_2 \gtrsim -\frac{1}{1-\left(\frac{r_1}{r_2}\right)^{m}}\int_{r_3}^{r_{max}}\left(\frac{r}{\rho}\right)^{m-1} m\eta(\rho)d\rho \gtrsim -m\frac{\left(\frac{\pi}{m}-\eta(r)\right)}{1-\left(\frac{r_1}{r_2}\right)^m}\int_{r_{3}}^{r_{max}}\left(\frac{r}{\rho}\right)^{m-1}d\rho \gtrsim  -\frac{rm}{(m-2)}\frac{\left(\frac{\pi}{m}-\eta(r)\right)}{1-\left(\frac{r_1}{r_2}\right)^m},
\end{align*}
where we used $\eta(\rho) < \frac{\pi}{m}-\eta(r)$ for $\rho>r_3$ to get the second inequality. Since $m\ge 3$ and $\frac{\pi}{m}-\eta(r)<\frac{\pi}{m}-\eta(r_2)=\delta$, the above inequality implies \begin{align}\label{K2bound}
K_2 \gtrsim -\frac{r}{1-\left(\frac{r_1}{r_2}\right)^m}\left(\frac{\pi}{m}-\eta(r)\right) \gtrsim -\frac{r}{1-\left(\frac{r_1}{r_2}\right)^m}\delta.
\end{align} Thus, \eqref{last112}, \eqref{K1bound} and \eqref{K2bound} yield \eqref{radial1}.\\

Now, let us prove \eqref{angular1}. Since $\sin(m\eta(\rho)) < \sin(m\delta)$ for all $\rho < r_2$, it follows from Lemma~\ref{derivatives1} that
\begin{align*}
\partial_\theta \varphi_m(r,\eta(r)) & <\frac{1}{4\pi}  \underbrace{\int_{r_{min}}^{r} \rho\log\left(1+\frac{4(\frac{\rho}{r})^m\sin^{2}(m\delta)}{(1-(\frac{\rho}{r})^m)^2}\right)d\rho}_{=:K_3} + \frac{1}{4\pi} \underbrace{\int_{r}^{r_2} \rho\log\left(1+\frac{4(\frac{r}{\rho})^m\sin^{2}(m\delta)}{(1-(\frac{r}{\rho})^m)^2}\right)d\rho}_{=:K_4}\\
& \ + \frac{1}{4\pi} \underbrace{\int_{r_2}^{r_{max}}\rho \log\left( 1 + \frac{4(\frac{r}{\rho})^m \sin(m\eta(\rho))\sin(m\delta)}{(1-(\frac{r}{\rho})^m)^2}\right)d\rho}_{=:K_5}.
\end{align*}
Again, the change of variables, $\left(\frac{\rho}{r}\right)^m\mapsto x$ and $\left(\frac{r}{\rho}\right)^m\mapsto x$ yields that
\begin{align*}
K_3 \lesssim \frac{r^2}{m}\int_0^{1} x^{-1+\frac{2}{m}}\log\left( 1+\frac{4x\sin^2(m\delta)}{(1-x)^2}\right)dx, \\
K_4 \lesssim \frac{r^2}{m}\int_0^{1} x^{-1-\frac{2}{m}}\log\left( 1+\frac{4x\sin^2(m\delta)}{(1-x)^2}\right)dx.
\end{align*}
Thus it follows from Lemma~\ref{inequality2} that
\begin{align}\label{k34bound}
K_3, \ K_4 \lesssim \frac{r^2}{m}\left|\sin(m\delta)\right| < r^2\delta.
\end{align}
To estimate $K_5$, recall that $r<r_1$ and $\log(1+x)\le x$, which yields that
\begin{align}
K_5 & \lesssim \int_{r_2}^{r_{max}}\rho\left( \frac{\left(\frac{r}{\rho}\right)^m\sin(m\eta(\rho))\sin(m\delta)}{\left(1-\left(\frac{r}{\rho}\right)^m\right)^2}\right) d\rho \nonumber\\
& = \frac{r^2}{m}\int_{\left(\frac{r}{r_{max}}\right)^m}^{\left(\frac{r}{r_2}\right)^m} x^{-1-\frac{2}{m}}\left(\frac{x\sin(m\eta(\rho))\sin(m\delta))}{(1-x)^2}\right)dx\nonumber\\
&\lesssim \frac{r^2}{m}\int_{0}^{\left(\frac{r_1}{r_2}\right)^m}x^{-\frac{2}{m}}\frac{\sin(m\delta)}{(1-x)^2}dx \nonumber \\
& = \frac{r^2}{m}\left(\int_0^{\min\left\{ \frac{1}{2},\left(\frac{r_1}{r_2}\right)^m\right\}}x^{-\frac{2}{m}}\frac{\sin(m\delta)}{(1-x)^2}dx +\int_{\min\left\{ \frac{1}{2},\left(\frac{r_1}{r_2}\right)^m\right\}}^{\left(\frac{r_1}{r_2}\right)^m}x^{-\frac{2}{m}}\frac{\sin(m\delta)}{(1-x)^2}dx \right)\nonumber\\
& \lesssim \frac{r^2}{m}\left( m\delta + \frac{\sin(m\delta)}{\left(1-\left(\frac{r_1}{r_2}\right)^m\right)}\right) \nonumber\\
&\lesssim \frac{\delta}{1-\left(\frac{r_1}{r_2}\right)^m}\label{k5bound}
\end{align}
Therefore \eqref{k34bound} and \eqref{k5bound} yield \eqref{angular1}. This finishes the proof.
\end{prooflem}

\appendix
\section{Appendix}\label{section5}
\subsection{Derivatives of the stream function}\label{ap1}
In this appendix, we will derive some formulae for zero-mean stream function by using Fourier series.
\begin{lemma}\label{streamfunction}
For $\rho>0$, let $h\in L^{2}(\partial B_\rho)$ such that $\int_{|y|=\rho}h(y)d\mathcal{H}^{1}(y) = 0$. Then it holds that for $x=\left(r\cos\theta,r\sin\theta\right)$,
\begin{align*}
\frac{1}{2\pi}\int_{|y|=\rho}h(y)\log|x-y|d\mathcal{H}^{1}(y) = 
\begin{cases}
-\sum_{n=-\infty n\ne 0}^{\infty} \frac{\rho}{2|n|}\hat{h}(\rho,n)\left(\frac{r}{\rho}\right)^{|n|}e^{in\theta}  & \text{ if } \rho\ge r, \\
-\sum_{n=-\infty n\ne 0}^{\infty} \frac{\rho}{2|n|}\hat{h}(\rho,n)\left(\frac{\rho}{r}\right)^{|n|}e^{in\theta}  & \text{ if }\rho < r,
\end{cases}
\end{align*}
where $\hat{h}(\rho,n) = \frac{1}{2\pi}\int_{\partial B} h(\rho y)e^{-iny}d\mathcal{H}^1(y)$.
\end{lemma}

\begin{proof}

By adapting the abuse of notation $h(y) = h(\rho,\eta)$, for $y=(\rho\cos\eta,\rho\sin\eta)$, we have that
\begin{align*}
\frac{1}{2\pi}\int_{|y|=\rho} h(y)\log|x-y|d\mathcal{H}^1(y)& = \frac{\rho}{4\pi}\int_{\mathbb{T}}h(\rho,\eta)\log|(r\cos\theta,r\sin\theta)-(\rho\cos\eta,\rho\sin\eta)|^2d\eta\\
&=\frac{\rho}{4\pi}\int_{\mathbb{T}}h(\rho,\eta) \log(r^2+\rho^2-2r\rho\cos(\theta-\eta))d\eta.
\end{align*}
Using the Fourier expansion $h(\rho,\eta) := \sum_{n=-\infty, n\ne0}^{\infty}\hat{h}(\rho,n)e^{in\eta}$ where $\hat{h}(\rho,n):=\frac{1}{2\pi}\int_{\mathbb{T}} h(\rho,\eta)e^{-in\eta}d\eta$, we have
\begin{align*}
\frac{\rho}{4\pi}\int_{\mathbb{T}}h(\rho,\eta) \log(r^2+\rho^2-2r\rho\cos(\theta-\eta))d\eta & = \sum_{n=-\infty, n\ne 0}^{\infty}\frac{\rho}{4\pi}\hat{h}(\rho,n)\underbrace{\int_{\mathbb{T}}e^{in\eta}\log(r^2+\rho^2-2r\rho\cos(\theta-\eta))d\eta}_{=:A_n(r,\rho,\theta)},
\end{align*}
where we used $\hat{h}(\rho,0)=0$ since $h$ has zero mean on $\partial B_\rho$. 
To compute $A_n$, we recall from \cite[Lemma A.1]{castro2019uniformly} that for $0\le x\le 1$ and $\mathbb{Z}\ni n\ne 0$, it holds that
\begin{align}\label{stream113}
\int_{\mathbb{T}} e^{in\eta}\log(1+x^2-2x\cos(\theta-\eta))d\eta = -\frac{2\pi}{|n|}e^{in\theta}x^{|n|}.
\end{align}
Then it directly follows from \eqref{stream113} that 
\begin{align*}
A_n(r,\rho,\theta) = 
\begin{cases}
-\frac{2\pi}{|n|}e^{in\theta}\left(\frac{r}{\rho}\right)^{|n|} & \text{ if }\rho\ge r\\
-\frac{2\pi}{|n|}e^{in\theta}\left(\frac{\rho}{r}\right)^{|n|} & \text{ if }\rho< r.
\end{cases}
\end{align*}
Plugging this into the above equation, the desired result follows immediately.

\end{proof}

\begin{lemma}\label{derivatives335}
For a bounded $m$-fold symmetric domain $D$ in $\R^2$, let us consider a decomposition of $1_D * \mathcal{N}$,
\begin{align*}
1_D * \mathcal{N}(r,\theta) = g * \mathcal{N}(r) + \left( 1_D - g \right) * \mathcal{N}(r,\theta) =: \varphi^{r}(r) + \varphi_{m}(r,\theta),
\end{align*}
where $g(r) := \frac{1}{2\pi r}\mathcal{H}^1\left(\partial B_r \cap D\right)$. Then, 
\begin{align}
\partial_r \varphi_m(r,\theta) & =\frac{1}{2\pi}\int_{\mathbb{T}}\int_0^r h(\rho,\eta+\theta)\left(\sum_{n=1}^{\infty}\left(\frac{\rho}{r}\right)^{nm+1}\cos(nm\eta)\right)d\rho d\eta \nonumber \\
&\ - \frac{1}{2\pi}\int_{\mathbb{T}}\int_{r}^{\infty}h(\rho,\eta+\theta)\left(\sum_{n=1}^{\infty}\left(\frac{r}{\rho}\right)^{nm-1}\cos(nm\eta)\right)d\rho d\eta\label{errorderivative111}\\
\partial_{\theta}\varphi_m(r,\theta) & =-\frac{r}{2\pi}\int_{\mathbb{T}}\int_0^r h(\rho,\eta+\theta)\left(\sum_{n=1}^{\infty}\left(\frac{\rho}{r}\right)^{nm+1}\sin(nm\eta)\right)d\rho d\eta \nonumber \\
&\ -  \frac{r}{2\pi} \int_{\mathbb{T}}\int_{r}^{\infty}h(\rho,\eta+\theta)\left(\sum_{n=1}^{\infty}\left(\frac{r}{\rho}\right)^{nm-1}\sin(nm\eta)\right)d\rho d\eta,\label{errorderivative112}
\end{align}
where $h(\rho,\theta):= 1_D(\rho(\cos\theta,\sin\theta))-g(\rho)$.

\end{lemma}

\begin{proof}
We compute that for $x:=(r\cos\theta,r\sin\theta)$,
\begin{align*}
\varphi_m (r,\theta) &= \frac{1}{2\pi}\int_{\R^2} \left( 1_D(y)-g(|y|)\right)\log|x-y|dy = \int_{0}^{\infty} \left( \frac{1}{2\pi}\int_{|y|=\rho} \underbrace{\left(1_D(y)-g(|y|)\right)}_{=:h(y)}\log|x-y|d\mathcal{H}^1(y)\right) d\rho.
\end{align*}
By adapting the abuse of notation $h(y) = h(\rho,\eta)$ for $y=(\rho\cos\eta,\rho\sin\eta)$, we have $\int_{\mathbb{T}}h(\rho,\eta)d\eta=0$ for all $\rho>0$. Since $D$ is $m$-fold symmetric, we also have that $\eta \mapsto h(\rho,\eta)$ is $\frac{2\pi}{m}$-periodic function for each fixed $\rho$.  Therefore, it follows from Lemma~\ref{streamfunction} that

\begin{align*}
\varphi_m(r,\theta) := -\sum_{n=-\infty, n\ne 0}^{\infty}\frac{1}{2|nm|}\left(\int_{0}^{r}\rho\hat{h}(\rho,nm)\left(\frac{\rho}{r}\right)^{|nm|}e^{inm\theta} d\rho +\int_r^{\infty} \rho\hat{h}(\rho,nm)\left(\frac{r}{\rho}\right)^{|nm|}e^{inm\theta} d\rho\right),
\end{align*}
where $\hat{h}(\rho,nm) :=\frac{1}{2\pi} \int_{\mathbb{T}}h(\rho,\eta)e^{-inm\eta}d\eta$. Therefore we have
\begin{align}
&\partial_r \varphi_m(r,\theta) = -\sum_{n=-\infty, n\ne 0}^{\infty} \frac{1}{2}\left( -\int_0^r \hat{h}(\rho,nm)\left(\frac{\rho}{r}\right)^{|nm|+1}e^{inm\theta}d\rho + \int_r^{\infty}\hat{h}(\rho,nm)\left(\frac{r}{\rho}\right)^{|nm|-1}e^{inm\theta}d\rho\right), \label{radial112}\\
&\partial_{\theta} \varphi_m(r,\theta) = -\sum_{n=-\infty, n\ne 0 }^{\infty} \frac{in}{2|n|}r\left( \int_0^r \hat{h}(\rho,nm)\left(\frac{\rho}{r}\right)^{|nm|+1}e^{inm\theta} d\rho + \int_r^{\infty} \hat{h}(\rho,nm)\left(\frac{r}{\rho}\right)^{|nm|-1}e^{inm\theta} d\rho\right).\label{angular112}
\end{align}
To simplify the radial derivative, we use the definition of $\hat{h}$ and \eqref{radial112} to obtain
\begin{align*}
\partial_r \varphi_m(r,\theta) & =  \frac{1}{4\pi}\int_{\mathbb{T}}\int_{0}^{r}h(\rho,\eta)\left(\sum_{n\ne 0 }\left(\frac{\rho}{r}\right)^{|nm|+1}e^{inm(\theta-\eta)}\right)d\rho d\eta\\
& \  - \frac{1}{4\pi}\int_{\mathbb{T}}\int_{r}^{\infty}h(\rho,\eta)\left(\sum_{n\ne 0 }\left(\frac{r}{\rho}\right)^{|nm|-1}e^{inm(\theta-\eta)}\right)d\rho d\eta\\
&= \frac{1}{4\pi}\int_{\mathbb{T}}\int_{0}^{r}h(\rho,\eta+\theta)\left(\sum_{n\ne 0 }\left(\frac{\rho}{r}\right)^{|nm|+1}e^{-inm\eta}\right)d\rho d\eta\\
& \ -  \frac{1}{4\pi}\int_{\mathbb{T}}\int_{r}^{\infty}h(\rho,\eta+\theta)\left(\sum_{n\ne 0 }\left(\frac{r}{\rho}\right)^{|nm|-1}e^{-inm\eta}\right)d\rho d\eta\\
&=\frac{1}{2\pi}\int_{\mathbb{T}}\int_0^r h(\rho,\eta+\theta)\left(\sum_{n=1}^{\infty}\left(\frac{\rho}{r}\right)^{nm+1}\cos(nm\eta)\right)d\rho d\eta\\
&\ - \frac{1}{2\pi}\int_{\mathbb{T}}\int_{r}^{\infty}h(\rho,\eta+\theta)\left(\sum_{n=1}^{\infty}\left(\frac{r}{\rho}\right)^{nm-1}\cos(nm\eta)\right)d\rho d\eta,
\end{align*}
where we used the change of variables, $\eta \mapsto \eta+\theta$ to get the second inequality. This proves \eqref{errorderivative111}. In the same way, we use \eqref{angular112} and the change of variables to obtain
\begin{align*}
\partial_{\theta}\varphi_m(r,\theta) & = -\frac{r}{4\pi}\int_{\mathbb{T}}\int_{0}^{r}h(\rho,\eta)\left(\sum_{n\ne 0 }\frac{in}{|n|}\left(\frac{\rho}{r}\right)^{|nm|+1}e^{inm(\theta-\eta)}\right)d\rho d\eta\\
& \  - \frac{r}{4\pi}\int_{\mathbb{T}}\int_{r}^{\infty}h(\rho,\eta)\left(\sum_{n\ne 0 }\frac{in}{|n|}\left(\frac{r}{\rho}\right)^{|nm|-1}e^{inm(\theta-\eta)}\right)d\rho d\eta\\
&= -\frac{r}{4\pi}\int_{\mathbb{T}}\int_{0}^{r}h(\rho,\eta+\theta)\left(\sum_{n\ne 0 }\frac{in}{|n|}\left(\frac{\rho}{r}\right)^{|nm|+1}e^{-inm\eta}\right)d\rho d\eta\\
& \  - \frac{r}{4\pi}\int_{\mathbb{T}}\int_{r}^{\infty}h(\rho,\eta+\theta)\left(\sum_{n\ne 0 }\frac{in}{|n|}\left(\frac{r}{\rho}\right)^{|nm|-1}e^{-inm\eta}\right)d\rho d\eta\\
&=-\frac{r}{2\pi}\int_{\mathbb{T}}\int_0^r h(\rho,\eta+\theta)\left(\sum_{n=1}^{\infty}\left(\frac{\rho}{r}\right)^{nm+1}\sin(nm\eta)\right)d\rho d\eta\\
&\ - \frac{r}{2\pi}\int_{\mathbb{T}}\int_{r}^{\infty}h(\rho,\eta+\theta)\left(\sum_{n=1}^{\infty}\left(\frac{r}{\rho}\right)^{nm-1}\sin(nm\eta)\right)d\rho d\eta,
\end{align*}
which proves \eqref{errorderivative112}.

\end{proof}

\begin{lemma}\label{derivatives1}
For a patch $D$ that satisfies the assumptions (a)-(c) in subsection~\ref{proofoftheorem3}, it holds that for $r\in (r_{min},r_{max})$ and $\eta := u^{-1}$, 
\begin{align*}
\partial_r \varphi_m(r,\theta) = \int_{r_{min}}^{r_{max}}f_1(\rho,r,\theta)d\rho\\
\partial_\theta \varphi_m(r,\theta) = \int_{r_{min}}^{r_{max}}f_2(\rho,r,\theta)d\rho,
\end{align*}
where
\begin{align}
& f_1(\rho,r,\theta) = 
\begin{cases}
\frac{1}{2\pi} \frac{\rho}{r} \left(\arctan{\left(\frac{(\frac{r}{\rho})^m\sin{(m(\theta-\eta(\rho)))}}{1-\left(\frac{r}{\rho}\right)^m\cos{(m(\theta-\eta(\rho)))}}\right)}- \arctan{\left(\frac{(\frac{r}{\rho})^m\sin{(m(\theta+\eta(\rho)))}}{1-\left(\frac{r}{\rho}\right)^m\cos{(m(\theta+\eta(\rho)))}}\right)}\right) & \text{ if }\rho \ge r \\
 \frac{1}{2\pi} \frac{\rho}{r} \left(\arctan{\left(\frac{(\frac{\rho}{r})^m\sin{(m(\theta+\eta(\rho)))}}{1-\left(\frac{\rho}{r}\right)^m\cos{(m(\theta+\eta(\rho)))}}\right)}- \arctan{\left(\frac{(\frac{\rho}{r})^m\sin{(m(\theta-\eta(\rho)))}}{1-\left(\frac{\rho}{r}\right)^m\cos{(m(\theta-\eta(\rho)))}}\right)}\right) & \text{ if }\rho < r, \label{f1}
\end{cases}\\
& f_2(\rho,r,\theta) = 
\begin{cases}
\frac{\rho}{4\pi}\log\left(1+\frac{4(\frac{r}{\rho})^m\sin(m\eta(\rho))\sin(m\theta)}{1+(\frac{r}{\rho})^{2m}-2(\frac{r}{\rho})^m\cos(m(\theta-\eta(\rho)))}\right) & \text{ if }\rho \ge r \\
\frac{\rho}{4\pi}\log\left(1+\frac{4(\frac{\rho}{r})^m\sin(m\eta(\rho))\sin(m\theta)}{1+(\frac{\rho}{r})^{2m}-2(\frac{\rho}{r})^m\cos(m(\theta-\eta(\rho)))}\right) & \text{ if }\rho < r. \label{f2}
\end{cases}
\end{align} 
\end{lemma}

\begin{proof} The proof is based on Lemma~\ref{derivatives335}. Using $m$-fold symmetry and evenness of the patch, we will compute the series. 

 From Lemma~\ref{derivatives335} and Fubini theorem, it follows that
\begin{align}\label{eqn551}
\partial_r \varphi_m(r,\theta) & =\frac{1}{2\pi}\sum_{n=1}^{\infty}\int_0^r \left(\frac{\rho}{r}\right)^{nm+1}\left(\int_{\mathbb{T}}h(\rho,s+\theta)\cos(nms)ds \right)d\rho  \nonumber \\
&\ - \frac{1}{2\pi}\sum_{n=1}^{\infty}\int_r^{\infty} \left(\frac{r}{\rho}\right)^{nm-1}\left(\int_{\mathbb{T}}h(\rho,s+\theta)\cos(nms)ds \right)d\rho,
\end{align}
and
\begin{align}\label{eqn552}
\partial_\theta \varphi_m(r,\theta) & =-\frac{r}{2\pi}\sum_{n=1}^{\infty}\int_0^r \left(\frac{\rho}{r}\right)^{nm+1}\left(\int_{\mathbb{T}}h(\rho,s+\theta)\sin(nms)ds \right)d\rho  \nonumber \\
&\ - \frac{r}{2\pi}\sum_{n=1}^{\infty}\int_r^{\infty} \left(\frac{r}{\rho}\right)^{nm-1}\left(\int_{\mathbb{T}}h(\rho,s+\theta)\sin(nms)ds \right)d\rho,
\end{align}
where $h(\rho,s) := 1_D(\rho\cos{s},\rho\sin{s})-g(\rho)$. Using the definition of $\eta = u^{-1}$, the following holds for $s\in [ -\frac{\pi}{m},\frac{\pi}{m} ]$:
\begin{align*}
h(\rho,s) = 
\begin{cases}
1-g(\rho) & \text{ if }s\in (-\eta(\rho),\eta(\rho)), \\
-g(\rho) & \text{ if }s \in [-\frac{\pi}{m},\frac{\pi}{m}]\backslash (-\eta(\rho),\eta(\rho)).
\end{cases}
\end{align*}
Therefore $m$-fold symmetry of $D$ yields that
\begin{align}\label{costerm1}
\int_{\mathbb{T}}h(\rho,s+\theta)\cos(nms)ds & = m\int_{-\frac{\pi}{m}}^{\frac{\pi}{m}}h(\rho,s)\cos(nm(s-\theta))ds\nonumber\\
& = m\int_{-\eta(\rho)}^{\eta(\rho)}\cos(nm(s-\theta))ds\nonumber\\
&=\frac{1}{n}\left(\sin(nm(\eta(\rho)-\theta))+\sin(nm(\eta(\rho)+\theta))\right).
\end{align}
Similarly, we have
\begin{align}\label{sinterm1}
\int_{\mathbb{T}}h(\rho,s+\theta)\sin(nms)ds = -\frac{1}{n}\left(\cos(nm(\eta(\rho)-\theta))-\cos(nm(\eta(\rho)+\theta))\right)
\end{align}
 
 Hence \eqref{eqn551} and \eqref{costerm1} yield that
 \begin{align*}
 \partial_r & \varphi_m(r,\theta) \\
 & = \frac{1}{2\pi}\int_{0}^{r}\sum_{n=1}^{\infty}\left(\frac{1}{n}\left(\frac{\rho}{r}\right)^{nm+1}\left(\sin(nm(\eta(\rho)-\theta))+\sin(nm(\eta(\rho)+\theta))\right)\right)d\rho\\
 & \ - \frac{1}{2\pi}\int_{r}^{\infty}\sum_{n=1}^{\infty}\left(\frac{1}{n}\left(\frac{r}{\rho}\right)^{nm-1}\left(\sin(nm(\eta(\rho)-\theta))+\sin(nm(\eta(\rho)+\theta))\right)\right)d\rho\\
 & =\frac{1}{2\pi}\int_{0}^{r}\frac{\rho}{r} \left(\arctan\left( \frac{\left(\frac{\rho}{r}\right)^m\sin(m(\eta(\rho)-\theta))}{1-\left(\frac{\rho}{r}\right)^m\cos(m(\eta(\rho)-\theta))}\right) + \arctan\left( \frac{\left(\frac{\rho}{r}\right)^m\sin(m(\eta(\rho)+\theta))}{1-\left(\frac{\rho}{r}\right)^m\cos(m(\eta(\rho)+\theta))}\right) \right)d\rho\\
 &\ -\frac{1}{2\pi}\int_{r}^{\infty}\frac{\rho}{r} \left(\arctan\left( \frac{\left(\frac{r}{\rho}\right)^m\sin(m(\eta(\rho)-\theta))}{1-\left(\frac{r}{\rho}\right)^m\cos(m(\eta(\rho)-\theta))}\right) + \arctan\left( \frac{\left(\frac{r}{\rho}\right)^m\sin(m(\eta(\rho)+\theta))}{1-\left(\frac{r}{\rho}\right)^m\cos(m(\eta(\rho)+\theta))}\right) \right)d\rho,
 \end{align*}
where the last equality follows from \eqref{formula1} in Lemma~\ref{series3}. Since the integrands in the above integrals are zero if $\rho<r_{min}$ or $\rho>r_{max}$, we can replace $0$ and $\infty$ in integration limits  by $r_{min}$ and $r_{max}$, respectively.   This proves \eqref{f1}. To prove \eqref{f2}, we use \eqref{eqn552},\eqref{sinterm1} and \eqref{formula2} to obtain
\begin{align*}
\partial_\theta &\varphi_m(r,\theta) \\
& = \frac{r}{2\pi}\int_{0}^{r}\sum_{n=1}^{\infty}\left(\frac{1}{n}\left(\frac{\rho}{r}\right)^{nm+1}\left(\cos(nm(\eta(\rho)-\theta))-\cos(nm(\eta(\rho)+\theta))\right)\right)d\rho\\
 & \ + \frac{r}{2\pi}\int_{r}^{\infty}\sum_{n=1}^{\infty}\left(\frac{1}{n}\left(\frac{r}{\rho}\right)^{nm-1}\left(\cos(nm(\eta(\rho)-\theta)) - \cos(nm(\eta(\rho)+\theta))\right)\right)d\rho\\
  &= \frac{r}{4\pi}\int_{0}^{r}\frac{\rho}{r}  \log\left( \frac{1+\left(\frac{\rho}{r}\right)^{2m}-2\left(\frac{\rho}{r}\right)^m\cos(m(\eta(\rho)+\theta))}{1+\left(\frac{\rho}{r}\right)^{2m}-2\left(\frac{\rho}{r}\right)^m\cos(m(\eta(\rho)-\theta))}\right)d\rho\\
 & \ + \frac{r}{4\pi}\int_r^\infty \frac{\rho}{r} \log\left( \frac{1+\left(\frac{r}{\rho}\right)^{2m}-2\left(\frac{r}{\rho}\right)^m\cos(m(\eta(\rho)+\theta))}{1+\left(\frac{r}{\rho}\right)^{2m}-2\left(\frac{r}{\rho}\right)^m\cos(m(\eta(\rho)-\theta))}\right)d\rho\\
 &= \frac{1}{4\pi}\int_{0}^{r}\rho  \log\left(1+ \frac{2\left(\frac{\rho}{r}\right)^m(\cos(m(\eta(\rho)-\theta))-\cos(m(\eta(\rho)+\theta)))}{1+\left(\frac{\rho}{r}\right)^{2m}-2\left(\frac{\rho}{r}\right)^m\cos(m(\eta(\rho)-\theta))}\right) d\rho\\
 &\ + \frac{1}{4\pi}\int_r^\infty \rho \log\left( 1+\frac{2\left(\frac{r}{\rho}\right)^m(\cos(m(\eta(\rho)-\theta))-\cos(m(\eta(\rho)+\theta)))}{1+\left(\frac{r}{\rho}\right)^{2m}-2\left(\frac{r}{\rho}\right)^m\cos(m(\eta(\rho)-\theta))}\right)d\rho\\
 &= \frac{1}{4\pi} \int_0^r \rho\log\left(1+\frac{4(\frac{\rho}{r})^m\sin(m\eta(\rho))\sin(m\theta)}{1+(\frac{\rho}{r})^{2m}-2(\frac{\rho}{r})^m\cos(m(\theta-\eta(\rho)))}\right) d\rho \\
 &\ + \frac{1}{4\pi}\int_r^\infty \rho \log\left(1+\frac{4(\frac{r}{\rho})^m\sin(m\eta(\rho))\sin(m\theta)}{1+(\frac{r}{\rho})^{2m}-2(\frac{r}{\rho})^m\cos(m(\theta-\eta(\rho)))}\right) d\rho,
\end{align*}
where the last equality follows from $\cos(x-y)-\cos(x+y)=2\sin x \sin y$. This proves \eqref{f2}.
\end{proof}

\subsection{Helpful lemmas}

\begin{lemma}\label{series3}
For $|x| < 1$ and $y\in (-\pi,\pi)$, it holds that
\begin{align}
\sum_{n=1}^{\infty}\frac{1}{n}x^{n} \cos(ny) = -\frac{1}{2}\log(1+x^2-2x\cos y), \label{formula2} \\
\sum_{n=1}^{\infty}\frac{1}{n}x^n \sin(ny) = \arctan{\left( \frac{x\sin y}{1-x\cos y}\right)}. \label{formula1}
\end{align}
Consequently, we have
\begin{align}
\sum_{n=1}^{\infty}x^n\cos(ny) = \frac{x(\cos{y}-x)}{(1-x)^2+2x(1-\cos{y})},\label{formula4} \\
\sum_{n=1}^{\infty}x^n\sin(ny) = \frac{x\sin{y}}{(1-x)^2+2x(1-\cos{y})}.\label{formula3}
\end{align}
\end{lemma}

\begin{proof}
Let $f(x,y):= \sum_{n=1}^{\infty}\frac{1}{n}x^ne^{iny}$. Then we compute
\begin{align*}
\partial_x f(x,y) & = \frac{1}{x}\sum_{n=1}^{\infty}\left(xe^{iy}\right)^{n}=\frac{e^{iy}}{1-xe^{iy}} = \frac{(\cos y - x) + i  \sin y}{(1-x\cos y)^2 + x^2 \sin^2 y}\\
& = \partial_x\left(-\frac{1}{2}\log(1+x^2-2x\cos y) + i \arctan\left(\frac{x\sin y}{1-x\cos y}\right)\right).
\end{align*}
Since $f(0,y) = 0$, we have $f(x,y) = -\frac{1}{2}\log(1+x^2-2x\cos y) + i \arctan\left(\frac{x\sin y}{1-x\cos y}\right)$. Equating the real and imaginary parts separately, we can obtain \eqref{formula2} and \eqref{formula1}. By differentiating \eqref{formula2} and \eqref{formula1} and multiplying by $x$, one can easily obtain \eqref{formula4} and \eqref{formula3}.
\end{proof}

\begin{lemma}\label{inequality1} 
For $m\ge3$ and $a,b\in (0,1)$, it holds that
\begin{align*}
\int_{0}^{1} x^{-1-\frac{2}{m}} \left(\arctan\left( \frac{ax}{1-x}\right)-\arctan\left(\frac{ax}{1-bx}\right)\right)dx \lesssim 1-b
\end{align*}
\end{lemma}
\begin{proof}
 By the change of variables, $bx\mapsto x$, we have
\begin{align*}
\int_0^1 x^{-1-\frac{2}{m}} \arctan\left( \frac{a x}{1-bx}\right)dx & = \int_0^{b} b^{\frac{2}{m}}x^{-1-\frac{2}{m}}\arctan\left(\frac{\frac{a x}{b}}{1-x}\right)dx\\
& \ge \int_0^{b}bx^{-1-\frac{2}{m}}\arctan\left(\frac{a x}{1-x}\right)dx,
\end{align*}
where we used $b^{\frac{2}{m}}\ge b$ for $0<b<1$ and $m\ge 3$.
Therefore it follows that
\begin{align*}
\int_{0}^{1} x^{-1-\frac{2}{m}} \left(\arctan\left( \frac{ax}{1-x}\right)-\arctan\left(\frac{ax}{1-bx}\right)\right)dx & \le \int_{b}^{1} x^{-1-\frac{2}{m}}\arctan\left(\frac{a x}{1-x}\right)dx\\ 
& \ +(1-b)\int_0^{b}x^{-1-\frac{2}{m}} \arctan\left(\frac{a x}{1-x}\right)dx\\
& \lesssim 1-b,
\end{align*}
which proves the desired inequality.
\end{proof}

\begin{lemma}\label{inequality2}
For $m\ge 3$ and $a\in(0,1)$, it holds that
\begin{align}
\int_0^{1}x^{-1-\frac{2}{m}}\log\left( 1+\frac{ax}{(1-x)^2}\right)dx \lesssim \sqrt{a}.\end{align}
\end{lemma}

\begin{proof}
If $x<\frac{1}{2}$, then $\log(1+\frac{ax}{(1-x)^2}) \lesssim ax$. Therefore,
\begin{align}\label{equation112312}
\int_0^{1}x^{-1-\frac{2}{m}}\log\left( 1+\frac{ax}{(1-x)^2}\right)dx & \lesssim \int_0^\frac{1}{2}ax^{-\frac{2}{m}}dx +\int_{\frac{1}{2}}^{1}\log\left(1+\frac{ax}{(1-x)^2}\right)dx \nonumber\\
& \lesssim a + \int_{\frac{1}{2}}^{1}\log\left(1+\frac{a}{(1-x)^2}\right)dx,
\end{align}
where we used $m\ge 3$ to estimate the first integral and $x\in (\frac{1}{2},1)$ for the second integral. To estimate the second integral, we compute
\begin{align}\label{equation112313}
\int_{\frac{1}{2}}^{1}\log \left( 1+\frac{a}{(1-x)^2}\right)dx &= \int_{\frac{1}{2}}^1 \frac{d}{dx}(x-1)\log\left(1+\frac{a}{(1-x)^2}\right)dx \nonumber\\
& =\frac{1}{2}\log(1+4a) + \int_{\frac{1}{2}}^1 \frac{2a}{(1-x)^2+a}dx\nonumber\\
& \lesssim a + \int_{\frac{1}{2}}^{1-\sqrt{a}} \frac{a}{(1-x)^2}dx + \int_{1-\sqrt{a}}^{1}1dx\nonumber\\
&\lesssim \sqrt{a}.
\end{align}
Thus the desired result follows from \eqref{equation112312} and \eqref{equation112313}.
\end{proof}

\section*{Acknowledgments.}
The author would like to thank his academic advisor, Prof. Yao Yao, for suggesting the problem and reading the draft paper. The author was partially supported by the NSF grants DMS-1715418 and DMS-1846745.

  \bibliographystyle{abbrv}
  
\bibliography{references}
\begin{tabular}{ll}
\textbf{Jaemin Park} \\
{\small School of Mathematics, Georgia Tech} \\
{\small 686 Cherry Street, Atlanta, GA 30332} \\
{\small Email: jpark776@gatech.edu} 
\end{tabular}

\end{document}